\documentclass[b5paper,10pt]{article}
\usepackage{amsmath,amsfonts,amssymb}
\usepackage{amsthm}
\usepackage{bm}
\usepackage[utf8]{inputenc}
\usepackage{url}
\usepackage{tikz}
\usepackage{hyperref}
\hypersetup{%
colorlinks=true,
linkcolor=black,
citecolor=black,
pdfauthor={Martin Sprengel, Gabriele Ciaramella},
pdfsubject={},
pdfkeywords={}
}

\newcommand{\R}{ \mathbb{R} }

\newcommand{\C}{ \mathbb{C} }
\newcommand{\dd}{\mathrm{d}}
\newcommand{\Span}{\text{span}}
\newcommand{\abl}[2]{\frac{\dd #1}{\dd #2}}
\newcommand{\pa}{\partial}
\newcommand{\pabl}[2]{\frac{\pa #1}{\pa #2}}
\newcommand{\dualP}[2]{\langle #1,\,#2 \rangle}
\newcommand{\scalar}[2]{\left(#1,\,#2\right)_{L^2}}

\newcommand{\scalarHOne}[2]{\left(#1,\,#2\right)_{H^1}}

\newcommand{\scalarC}[2]{\left(#1,\,#2\right)_{\C}}
\renewcommand{\Im}{\operatorname{Im}}
\renewcommand{\Re}{\operatorname{Re}}
\newcommand{\co}[1]{\overline{#1}}

\newtheorem{theorem}{Theorem}
\newtheorem{lemma}{Lemma}
\newtheorem{assumption}{Assumption}

\newtheorem{remark}{Remark}

\begin{document}
\title{A theoretical investigation of time-dependent Kohn-Sham equations\thanks{Supported in part by the Deutsche Forschungsgemeinschaft (DFG) project ``Controllability and Optimal Control of Interacting Quantum Dynamical Systems'' (COCIQS).}}
\author{
M.~Sprengel\thanks{Institut f\"ur Mathematik, Universit\"at W\"urzburg, Emil-Fischer-Strasse 30,
97074 W\"urzburg, Germany ({\tt martin.sprengel@mathematik.uni-wuerzburg.de}).}
         \and
G. Ciaramella\thanks{
Section de mathématiques, 
Université de Genève, 2-4 rue du Lièvre
1211 Genève 4, Switzerland ({\tt gabriele.ciaramella@unige.ch}).}
         \and 
A.~Borz{\`i}\thanks{Institut f\"ur Mathematik, Universit\"at W\"urzburg, Emil-Fischer-Strasse 30,
97074 W\"urzburg, Germany ({\tt alfio.borzi@mathematik.uni-wuerzburg.de}).}
       }
\date{}
\maketitle

\begin{abstract}
	In this work, the existence, uniqueness and regularity of solutions to the time-dependent Kohn-Sham equations are investigated. The Kohn-Sham equations are a system of nonlinear coupled Schrö\-din\-ger equations that describe multi-particle quantum systems in the framework of the time dependent density functional theory.
	In view of applications with control problems, the presence of a control function and of an inhomogeneity are also taken into account.
\end{abstract}

\section{Introduction}
The time-dependent density functional theory (TDDFT) was introduced to model multi-particle quantum systems avoiding the solution of the full Schrödinger equation (SE) in multi-dimensions \cite{KohnSham1965, ParrYang1989,  RuggenthalerPenzVanLeeuwen2015, RungeGross1984}.

The central concept of TDDFT is to describe the configuration of a $N$-particle quantum system using the density function $\rho$ that depends on the $n$-dimensional physical space coordinates and time. This is in contrast to the wave function representation of the full Schrödinger problem where $n\cdot N$ space coordinates are involved.

Specifically, in the TDDFT framework a system of $N$ nonlinear SEs is considered that governs the evolution of $N$ single-particle wave functions $\Psi=(\psi_1,\dotsc, \psi_N)$, $\psi_i=\psi_i(x,t)$, $x\in \R^n$, $t\in \R$. These SEs are coupled through a potential that depends on the density $\rho(x,t)=\sum_{i=1}^N |\psi_i(x,t)|^2$.
This time-dependent Kohn-Sham (TDKS) system is given by
\begin{equation}\label{eq:KS}
i \partial_t \Psi(x,t)
= I_N \otimes \bigl[- \nabla^2 +  V_{ext}(x,t;u) + V(x,t;\Psi) \bigr] \Psi(x,t), \quad \Psi(x,0)=\Psi_0(x)
\end{equation}
where $\nabla^2$ is the Laplacian, $V_{ext}$ is an external potential that includes the confining potential, e.g., the surrounding walls or the Coulomb potential of the nuclei of a molecule, and, possibly, a control potential. $V$ denotes the coupling $KS$ potential. See  \cite{RuggenthalerPenzVanLeeuwen2015} for a review on this model. 

The purpose of our work is to theoretically investigate \eqref{eq:KS}, with a given control function $u\in H^1(0,T)$, and an adjoint version of \eqref{eq:KS} that appears in the following optimal control problem
\begin{equation}\label{Intro:ControlProblem}
\begin{split}
	\min_{(\Psi,u)\in (W,H^1(0,T))}	J_1(\Psi)+J_2(\Psi(T)) + \nu\|u\|_{H^1(0,T)}^2 \quad \text{ s.t. } \Psi \text{ solves } \eqref{eq:KS},
\end{split}\end{equation}
where $\nu>0$ is a weight parameter and $J_1$ depends on the whole solution $\Psi$, while $J_2$ depends on the wave function at the final time $\Psi(T)$ only. The functionals $J_1, J_2$ are assumed to be lower semicontinuous and Fréchet differentiable with respect to $\Psi$.

To characterize the solutions to \eqref{Intro:ControlProblem} using the adjoint method \cite{Borzi2012}, the following adjoint equation is considered.
\begin{equation}\label{eq:KSadjoint}
\begin{split}
	i\pabl{\Psi}{t}&=I_N \otimes\left( -\nabla^2 +V_{ext}(x,t,u)
	+V(\Lambda)\right)\Psi\\
	 &+ I_N \otimes \left(V_H(2\Re\scalarC{\Psi}{\Lambda})+2 \pabl{V_{xc}}{\rho}(\Lambda)\Re\scalarC{\Psi}{\Lambda}+D_\psi J_1(\Lambda) \right) \Lambda,\\
	 \Psi(T)&=-D_\Psi J_2 (\Lambda(T)),
\end{split}
\end{equation}
where we again denote by $\Psi$ the adjoint variable while  the solution to \eqref{eq:KS} is denoted with $\Lambda$.
We remark that \eqref{eq:KSadjoint} has a similar structure as \eqref{eq:KS} with an additional inhomogeneity resulting from the Fréchet derivative $D_\psi J_1(\Lambda) \Lambda$ of $J_1$ with respect to the wave function, as well as additional terms resulting from the linearization of the Kohn-Sham potential. On the other hand, $V$ now depends on $\Lambda$ and is no longer a function of the unknown variables.
The derivative of $J_2$ gives a terminal condition for \eqref{eq:KSadjoint} that evolves backwards in time.

In this paper, we theoretically analyse \eqref{eq:KS} and \eqref{eq:KSadjoint} as two particular instances of a generalized TDKS equation, proving existence and uniqueness of solutions. At the best of our knowledge, this problem is only addressed in \cite{Jerome2015} for \eqref{eq:KS}. In this reference, the Author proves existence and uniqueness of solutions assuming that the Hamiltonian is  continuously differentiable in time. We improve these results in such a way that this theory can accommodate TDKS optimal control problems. In particular, existence and uniqueness of solutions with similar regularity as in \cite{Jerome2015} are proved also in the case when the external potential is only $H^1$ and not $C^1$. These results are achieved in the Galerkin framework.
We remark that by this approach, we address the TDKS equation \eqref{eq:KS} and its adjoint \eqref{eq:KSadjoint} in an unique framework. 
Notice that the adjoint problem has a different structure that can make it difficult the use of semi-group theory. 

This paper is organized as follows. In Section \ref{sec:model}, we discuss the KS potential $V$ and the external potential $V_{ext}$. Further, we formulate our evolution problem in a weak sense that embodies both \eqref{eq:KS} and \eqref{eq:KSadjoint}. Also in this section, we discuss the initial and boundary conditions, and provide specific assumptions on the potentials and the spatial domain $\Omega$ and the time interval $(0,T)$ where the KS problem is considered.
In Section \ref{sec:PreliminaryEstimates}, we investigate some properties of the KS potential and discuss continuity properties of the bilinear form resulting from the weak formulation. In Section \ref{sec:Galerkin}, we use the Galerkin framework  to obtain a finite dimensional approximation of our weak problem. In Section \ref{sec:EnergyEstimates}, we present energy estimates for the finite dimensional representation and their extension to the infinite dimensional case.
In Section \ref{sec:ExistenceSolution} and \ref{sec:Uniqueness}, we prove existence and uniqueness of solutions to our weak problem. First, we prove convergence of the Galerkin approximation to the infinite dimensional solution and use our results on the Lipschitz properties of the potential to prove uniqueness of this solution.
In Section \ref{sec:ImprovedRegularity}, assuming\linebreak $\Psi_0\in H^1_0(\Omega)$, we prove that the solution of our problem has higher regularity. The Sections \ref{sec:ExistenceSolution}, \ref{sec:Uniqueness}, and \ref{sec:ImprovedRegularity} present our main theoretical results. 
A section of conclusion completes this work. 

\section{The model description}
\label{sec:model}
In this section, we introduce the weak formulation of our evolution problem, define the potentials and discuss our assumptions.
To introduce the weak formulation of the evolution problem, we define the following function spaces. We use $L^2(\Omega;\C^N)$ where $\scalar{ \cdot }{ \cdot }$ is the scalar product  defined as follows
\begin{equation*}
\scalar{ \Psi }{ \Phi } := \int_{\Omega} \scalarC{ \Psi }{ \Phi } dx  ,
\end{equation*}
and $\| \cdot \|_{L^2}$ denotes the corresponding norm. Further, we denote by $\scalarC{ \cdot }{ \cdot }$ the scalar product for $\C^N$ and $| \cdot |$ is the corresponding norm. The scalar product of the Sobolev space $H^1(\Omega;\C^N)$ is given by
\begin{equation*}
\scalarHOne{ \Psi }{ \Phi } := \scalar{ \Psi }{ \Phi }
+\scalar{ \nabla \Psi }{ \nabla \Phi }  ,
\end{equation*}
and $\| \cdot \|_{H^1}$ denotes the corresponding norm.
Furthermore, the following spaces of functions of time and space with function values in $\C^N$ are used.\linebreak
$Y := L^2(0,T;L^2(\Omega;\C^N))$, $X:=L^2(0,T;H_0^1(\Omega; \C^N))$ with corresponding norms $\|u\|_Y^2=\int_0^T\|u(t)\|_{L^2}^2 \dd t$ and $\|u\|_X^2=\int_0^T\|u(t)\|_{H^1}^2 \dd t$ and its dual\linebreak $X^*=L^2(0,T; H^{-1}(\Omega; \C^N))$ and the space of solutions
$W := \{u \in X \text{ such that }\linebreak  u'\in  X^*\}$.

We prove the existence of a solution of the controlled Kohn-Sham model \eqref{eq:KS} and at the same time of \eqref{eq:KSadjoint} on a bounded domain $\Omega \subset \mathbb{R}^n$, $n = 3$, with homogeneous Dirichlet boundary conditions.
For this purpose, we denote by $\Psi \in X$ 
the vector of the wave functions corresponding to $N$ particles
\begin{equation}
\Psi := ( \: \psi_1, \dotsc, \psi_N \: )  ,
\end{equation}
and assume that $\psi_j(x,t) = 0$ for $x\in\partial \Omega$ and  consider the initial condition\linebreak $\psi_j(x,0) = \psi_{0,j}(x)$ with $\psi_{0,j} \in L^2(\Omega;\C)$.
Moreover, to include a possible inhomogeneity of the PDE, we consider the function $F \in Y$ defined as follows
\begin{equation}
F := ( \: f_1, \dotsc , f_N \: )  ,
\end{equation}
where $f_j \in L^2(0,T;L^2(\Omega;\C))$.

The wave function $\Psi$ gives rise to the density $\rho$ defined as follows
\begin{equation}
\rho(x,t) := \sum_j |\psi_j(x,t)|^2  ,
\end{equation}
which is used to characterize the nonlinear potential $V(x,t;\Psi)$. The dependence of $V$ on $\Psi$ is always through the density $\rho$, so we may also write  $V(x,t;\rho)$. In the local density approach (LDA) framework, $V$ is given by the sum of the Hartree, the exchange, and the correlation potentials. We have
\begin{equation}
\label{def:potential}
\begin{split}
V(x,t;\Psi) &= V_H+V_{xc}=V_H+V_x+V_c,\\
	V_H&=\int_\Omega \frac{\rho(y,t)}{|x-y|} \dd y, \quad V_x=V_x(\rho(x,t)), \quad V_c= V_c(\rho(x,t)).
\end{split}\end{equation}
$V_x$ is often derived from an approximation called the homogeneous electron gas \cite{ParrYang1989} and then given by $V_x= c \rho(x,t)^\beta$, where $c$ is a negative constant and \linebreak $0<\beta<1$ depends on the dimension $n$. For the correlation potential $V_c$ only numerical approximation exists. In the course of the years, physicists and quantum chemists have developed a collection of different $V_c$  functions. Similar to Jerome \cite{Jerome2015}, who uses a Lipschitz assumption on $V_x+V_c$, we make some general assumptions on the structure of the potentials rather than using an explicit form for one of the approximation used in applications.

The external potential is given by 
\begin{equation}
V_{ext}(x,t;u)=V_0(x)+V_u(x) u(t),
\end{equation}
where $V_0$ models a confinement potential, e.g., a harmonic trap in a solid state system or a molecule. The control potential $V_u(x) u(t)$ may represent a gate voltage applied to the solid state system or a laser pulse on the molecule.

We consider problems \eqref{eq:KS} and \eqref{eq:KSadjoint} in a unified framework by introducing a parameter $\alpha$ that indicates the case \eqref{eq:KS} by $\alpha=1$ and \eqref{eq:KSadjoint} by $\alpha=0$. The inhomogeneity $F$ can be zero as in \eqref{eq:KS} or given as in \eqref{eq:KSadjoint}. The equations are studied in the following weak form: 

Find a wave function $\Psi\in X$ with $\Psi'\in X^*$, such that
\begin{equation}\label{eq:KSweak}
\begin{split}
&i \scalar{ \partial_t \Psi(t)}{ \Phi }
= B(\Psi(t),\Phi;u(t)) + \alpha\scalar{ V(\Psi(t)) \Psi(t)}{ \Phi } + \scalar{ F(t) }{ \Phi } \\
&\text{a.e. in } (0,T) \text{ and } \forall \Phi \in H_0^1(\Omega; \C^N)  ,\\
&\Psi(0)=\Psi_0 \in L^2(\Omega; \C^N),
\end{split}
\end{equation}
where the bilinear form $B(\Psi,\Phi;u)$ is defined as follows
\begin{equation}\begin{split}
\label{eq:bilinearForm}
B(\Psi,\Phi;u) := &\scalar{ \nabla \Psi }{ \nabla \Phi }
+ \scalar{ V_{ext}(\cdot,\cdot;u) \Psi }{ \Phi }\\
&+(1-\alpha)\scalar{V(\cdot, \cdot, \Lambda)\Psi}{\Phi}+(1-\alpha) D(\Psi, \Phi) .
\end{split}\end{equation}
The additional terms of the adjoint equation are given by
\begin{equation}\label{eq:Dadjoint}
\begin{split}
	D(\Psi, \Phi)&=D_H(\Psi, \Phi)+D_{xc}(\Psi, \Phi),
\end{split}
\end{equation}
where
\begin{align*}
		D_H(\Psi, \Phi)&=\scalar{V_H(2\Re\scalarC{\Psi}{\Lambda})\Lambda}{\Phi},\\
		D_{xc}(\Psi, \Phi)&=\scalar{2 \pabl{V_{xc}}{\rho}(\Lambda)\Re\scalarC{\Psi}{\Lambda} \Lambda}{\Phi}.
\end{align*}

We remark that when studying the adjoint equation, the adjoint variable is also denoted with $\Psi$, and $\Lambda=(\lambda_1,\dotsc,\lambda_N)$ corresponds to the solution of the forward equation \eqref{eq:KS}. As we later prove in Theorem \ref{thm:ImprovedRegularity}, the solution of the forward equation $\Lambda$ is in $L^2((0,T); H^2(\Omega))$ and the embedding $H^2(\Omega)\hookrightarrow C(\bar{\Omega})$ guarantees that $\Lambda(t)$ is bounded a.e. in $(0,T)$, see, e.g., \cite[p. 332]{Ciarlet2013}.
Here,  $C(\bar{\Omega})$ is the space of continuous functions with the norm $\|f\|_{C}=\max_{k=1,\dotsc,N}\sup_{x\in \bar{\Omega}}|f_k(x)|$, 

In a quantum control setting, the inhomogeneity $F$ is zero in the forward equation and contains the derivative of $J_1$ with respect to the wave function in the adjoint equation. However, for generality we allow a nonzero $F$ when studying \eqref{eq:KS}. As in the argument above, $\Lambda$ and $J_1(\Lambda)$ are continuous functions of $x$ and hence in $L^2(\Omega)$.
To incorporate a final condition $\Psi(T)$ instead of an initial condition $\Psi(0)$, we substitute $t\mapsto T-t$.

Now, we want to summarize our assumptions that we make throughout our paper.
\begin{assumption}\label{assumptions}
We consider the following.
\begin{enumerate}
	\item A bounded domain $\Omega \subset \R^n$ with $n=3$ and a Lipschitz boundary; and, for the improved regularity in Theorem \ref{thm:ImprovedRegularity} and \ref{thm:ImprovedRegularity2}, $\pa \Omega\in C^2$.
	\item \label{assumptionVcBounded} The correlation potential  $V_c$ is uniformly bounded in the sense that\linebreak $|V_c(\Psi(x,t))|\leq K,$ $\forall x\in\Omega$, $t\in [0,T]$, $\Psi\in Y$; this is the case, e.g. for the Wigner potential \cite{Wigner1938}.
	\item \label{assumptionVxLipschitz} The exchange potential $V_x$ is Lipschitz continuous in the sense\linebreak  $\|V_x(\Psi)\Psi-V_x(\Upsilon)\Upsilon\|_{L^2}\leq L \|\Psi-\Upsilon\|_{L^2}$, for $\Psi$ and $\Upsilon$ being weak solutions of \eqref{eq:KSweak} and locally Lipschitz continuous for $\Psi\in L^2(\Omega)$, i.e. $\|V_x(\Psi)\Psi\|_{L^2}\leq L\|\Psi\|_{L^2}$, where $L$ might depend on $\|\Psi\|_{L^2}$; 
	cf. the similar assumption in \cite{Jerome2015}. This assumption will be motivated further in Remark \ref{remark:aposteriori}.
	\item \label{assumptionVxcDiffable} $V_x$ and $V_c$ are weakly differentiable as functions of $\rho$ and $\pabl{V_{xc}}{\rho}(\rho) \rho$ is bounded for finite values of $\rho$. For $V_x=c\rho^\beta$, this can be shown directly:
	\begin{align}
		\pabl{V_x(\rho)}{\rho}\rho=c\beta \rho^{\beta-1} \rho=c\beta\rho^\beta=\beta V_x(\rho).
	\end{align}
	\item The confining potential and the spacial dependence of the control potential are bounded, i.e. $V_0, V_u \in L^\infty(\Omega)$, where $\| f \|_{L^\infty}:=\max_{k=1,\dotsc,N}\operatorname*{ess\,sup}_{x\in \Omega} |f_k(x)|$ is the norm for $L^\infty(\Omega;\C^N)$; as we consider a finite domain, this is equivalent to excluding divergent external potentials.
	\item The control is $u \in H^1(0,T)$. This a classical assumption in optimal control, see, e.g. \cite{VonWinckelBorzi2008}.
	\item $\Psi_0\in L^2$ for existence and uniqueness of the forward and adjoint equations and $\Psi_0\in H^1$ for the improved regularity. For the adjoint equation, we assume that the solution of the forward problem  $\Lambda$ is in $L^2(0,T; H^2(\Omega; \C^N))$, this can be shown by applying Theorem \ref{thm:ImprovedRegularity} to the forward problem.
\end{enumerate}	
\end{assumption}

\section{Preliminary estimates}\label{sec:PreliminaryEstimates}
In this section, we study continuity properties of the KS potential and of the bilinear form.
We begin with a general result on the Coulomb potential $w(x)=\frac{1}{|x|}$. Then we investigate the continuity of the Hartree potential that is defined as the convolution of $w$ with the density $\rho$, and of the KS potential in more detail.
Finally, we prove some estimates for the bilinear forms $B$ and $D$.

\begin{lemma}\label{lem:CoulombInLp}
Given a bounded domain $\Omega \subset \R^n$ containing the origin, it holds that $w\in L^p(\Omega)$ if and only if $n>p$.
\end{lemma}
\begin{proof}
	By $B_R(0):=\{x\in \R^m: |x|<R\}$, we denote the open ball of radius $R\in\R^+$ around the origin.
	Consider now a ball $B_R(0) \subset\Omega$. Then by using spherical coordinates and the fact that $|x|$ does not depend on the orientation of $x$, we get (see e.g. \cite{Evans2010})
	\begin{align*}
		&\int_{B_R(0)} \frac{1}{|x|^p} \dd x 
		=\frac{n\pi^{n/2}}{\Gamma(\frac{n}{2}+1)}\int_0^R \frac{1}{r^p} r^{n-1} \dd r\\
		&=\frac{n\pi^{n/2}}{\Gamma(\frac{n}{2}+1)}\left[\frac{r^{n-p}}{n-p} \right]_0^R
		=\begin{cases}
		  	\frac{n\pi^{n/2}}{\Gamma(\frac{n}{2}+1)} \frac{R^{n-p}}{n-p} <\infty 	& n>p\\
		  	\infty							& n\leq p,
		  \end{cases}
	\end{align*}
	where $\Gamma$ is the $\Gamma$-function.
	Outside this ball, $\frac{1}{|x|}$ is globally bounded.
\end{proof}

\begin{lemma}\label{CancesLemma}
For $\Phi, \Psi\in H^1(\Omega; \C^N)$ there exists a positive constant $C_u$ such that
\begin{equation}
	\|V_H(\Phi)\Phi-V_H(\Psi)\Psi\|_{L^2} \leq C_u \left( \|\Phi\|_{H^1}^2+\|\Psi\|_{H^1}^2\right) \|\Phi-\Psi\|_{L^2}.
\end{equation}

\end{lemma}
\begin{proof}
	We adapt Lemma 5 in \cite{CancesLeBris1999} to our case of vector valued functions. To this end, we define $g^k(\Phi,\Psi)(x):=\int_\Omega \frac{\phi_k(y)\co{\psi_k(y)}}{|x-y|}\dd y$, $\tilde{g}(\Phi,\Psi)(x):=\sum_{k=1}^N g^k(\Phi,\Psi)(x)$. Then Lemma 3 in \cite{CancesLeBris1999} gives
	\begin{align*}
		|\tilde{g}(\Phi_1,\Phi_2)(x)|\leq \sum_{k=1}^N |g^k(\Phi_1,\Phi_2)(x)|\leq \sum_{k=1}^N \|\phi_{1,k}\|_{L^2}\|\nabla \phi_{2,k}\|_{L^2}.
	\end{align*}
	Using this fact and setting $\Upsilon=(v_1,\dotsc,v_N)$, we have
	\begin{align*}
		\|\tilde{g}(\Phi_1,\Phi_2)\Upsilon\|_{L^2}^2&=\sum_{k=1}^N\|\tilde{g}\upsilon_k\|_{L^2}^2=\sum_{k=1}^N \int_\Omega |\tilde{g}(x)\upsilon_k(x)|^2 \dd x\\
		&\leq\sum_{k=1}^N \left(\sum_{l=1}^N \|\phi_{1,l}\|_{L^2} \|\nabla \phi_{2,l}\|_{L^2}\right)^2 \int_\Omega |\upsilon_k(x)|^2\dd x\\
		&=\left(\sum_{l=1}^N \|\phi_{1,l}\|_{L^2} \|\nabla \phi_{2,l}\|_{L^2}\right)^2 \|\Upsilon\|_{L^2}^2\\
		&\leq N\sum_{l=1}^N \left(\|\phi_{1,l}\|_{L^2} \|\nabla \phi_{2,l}\|_{L^2}\right)^2 \|\Upsilon\|_{L^2}^2\\
		&\leq N\sum_{l=1}^N \|\phi_{1,l}\|_{L^2}^2 \sum_{j=1}^N\|\nabla \phi_{2,j}\|_{L^2}^2 \|\Upsilon\|_{L^2}^2\\
		&=N\|\Phi_1\|_{L^2}^2 \|\nabla \Phi_2\|_{L^2}^2\|\Upsilon\|_{L^2}^2.
	\end{align*}
	Now, we apply this to $\Phi_1=\Phi_2=\Phi$, $\Upsilon=\Phi-\Psi$ to obtain $\|(|\Phi|^2\star w)(\Phi-\Psi)\|_{L^2}\leq \sqrt{N}\|\Phi\|_{L^2}\|\nabla\Phi\|_{L^2}\|\Phi-\Psi\|_{L^2}$.
	Furthermore, using the decomposition
	\begin{align*}
		\sum_{k=1}^N|\phi_k|^2-\sum_{k=1}^N|\psi_k|^2=\sum_{k=1}^N\left(\phi_k \co{(\phi_k-\psi_k)}+\co{\psi_k}(\phi_k-\psi_k) \right),
	\end{align*}
	and $\Upsilon=\Psi$, $\Phi_1=\Phi-\Psi$, $\Phi_2=\Phi$ for the first term, and $\Upsilon=\Psi$, $\Phi_1=\Phi-\Psi$,$\Phi_2=\Psi$ for the second term,
	we find
	\begin{align*}
		\|((|\Phi|^2-|\Psi|^2)\star w) \Psi\|_{L^2}\leq \sqrt{N}\|\Psi\|_{L^2}(\|\nabla \Phi\|_{L^2}+\|\nabla \Psi\|_{L^2})\|\Phi-\Psi\|_{L^2}.
	\end{align*}
	With this, the proof of Lemma 5 in \cite{CancesLeBris1999} extends to the vector case.
\end{proof}

\begin{lemma}\label{lem:VContinuousFunction}
	The nonlinear KS potential $V$ is a continuous function from $L^2(\Omega)$ to $L^2(\Omega)$.
\end{lemma}
\begin{proof}
First, we show that $\rho(\Psi)$ is a continuous mapping from $L^2(\Omega)$ to $L^1(\Omega)$
	  in the sense that from $\Psi^n \stackrel{L^2}{\rightarrow} \hat{\Psi}$ follows $ \rho^n \stackrel{L^1}{\rightarrow} \hat{\rho}$. We have
	\begin{align*}
		\|\hat{\rho}-\rho_n\|_{L^1}&=\int_{\Omega}\left|\sum_j |\hat{\psi}_j|^2-\sum_j|\psi_j^n+\hat{\psi}_j-\hat{\psi}_j|^2 \right|\dd x\\
		&\leq \int_\Omega \left| \sum_j |\hat{\psi}_j-\psi_j^n|^2+2|\hat{\psi}_j-\psi_j^n|\,|\hat{\psi}_j| \right|\dd x\\
		&\leq\|\hat{\psi}-\psi^n\|_{L^2}^2+2\|\hat{\psi}-\psi^n\|_{L^2}\|\hat{\psi}\|_{L^2}\\
		&=3\max\{c,1\}\|\hat{\psi}-\psi^n\|_{L^2}^2,
	\end{align*}
	where $c:=\|\hat{\psi}\|_{L^2}$ and we use Cauchy-Schwarz inequality.
	
Second, $V_H$ is a continuous function of $\rho$ as follows
	\begin{align*}
		\|V_H(\rho_1)-V_H(\rho_2) \|_{L^2}&=\|w\star(\rho_1-\rho_2)\|_{L^2}\leq \|w\|_{L^2} \|\rho_1-\rho_2\|_{L^1},
	\end{align*}
	where Lemma \ref{lem:CoulombInLp} and Young's inequality \cite[Theorem 14.6]{Schilling2005} are used.
	
Finally, since $V_x$ and $V_c$ are pointwise differentiable as functions of $\rho$, they are also continuous.
\end{proof}

We continue with some estimates for the bilinear form $B$ for arbitrary wave functions.
\begin{lemma}
\label{lem:bound-weak-form}
There exist positive constants $c_0$, $c_1$, $c_2$ and $c_3$
such that the following estimates hold
\begin{equation}\label{eq:lem3:estD}
	|D(\Psi, \Phi)|\leq c_0 \|\Psi\|_{L^2} \|\Phi\|_{L^2}  ,
\end{equation}
\begin{equation}
\label{eq:est1}
\Re B(\Psi,\Phi;u) \leq |B(\Psi,\Phi;u)| \leq c_1 \| \Psi \|_{H^1} \| \Phi \|_{H^1}  ,
\end{equation}
\begin{equation}
\label{eq:est1im}
|\Im B(\Psi,\Psi;u)| \leq c_0 \| \Psi \|_{L^2}^2  ,
\end{equation}
and
\begin{equation}
\label{eq:est2}
 \| \Psi \|_{H^1}^2 \leq \Re B(\Psi,\Psi;u) + c_3 \| \Psi \|_{L^2}^2  ,
\end{equation}
for any $\Psi,\Phi \in H^1(\Omega;\C^N)$.
\end{lemma}
\begin{proof}
For $D(\Psi, \Phi)$ given by \eqref{eq:Dadjoint}, we use the fact that $\Lambda(\cdot, t)\in L^\infty(\Omega)$ and Assumption \ref{assumptions} (\ref{assumptionVxcDiffable}) to get
\begin{align*}
	|D_{xc}(\Psi, \Phi)| &=\left|\sum_{j=1}^N \sum_{k=1}^N \int_\Omega \pabl{V_{xc}}{\rho}(\Lambda) (\psi_j\co{\lambda_j}+\co{\psi_j}\lambda_j)\lambda_k \co{\phi_k} \dd x\right|\\
	&\leq 2\sum_{j=1}^N \sum_{k=1}^N \int_\Omega \left|\pabl{V_{xc}}{\rho}(\Lambda) \right| |\psi_j| \left(\sum_{l=1}^N |\lambda_l| \right)^2 |\phi_k| \dd x\\
	&\leq 2N\sum_{j=1}^N \sum_{k=1}^N \int_\Omega \left|\pabl{V_{xc}}{\rho}(\Lambda) \right| \left(\sum_{l=1}^N |\lambda_l|^2 \right)  |\psi_j| |\phi_k| \dd x\\
	&\leq 2N \sum_{j=1}^N \sum_{k=1}^N \left\|\pabl{V_{xc}}{\rho}(\Lambda)\sum_{l=1}^N |\lambda_l|^2 \right\|_{L^\infty}
	\int_\Omega (|\psi_j|+|\co{\psi_j}|) |\co{\phi_k}|\dd x\\
	&\leq 2N\left\|\pabl{V_{xc}}{\rho}(\Lambda)\sum_{l=1}^N |\lambda_l|^2\right\|_{L^\infty}
	\sum_{j=1}^N \sum_{k=1}^N  \|\psi_j\|_{L^2}\|\phi_k\|_{L^2}\\
	&\leq 2N^2\left\|\pabl{V_{xc}}{\rho}(\Lambda)\sum_{l=1}^N |\lambda_l|^2\right\|_{L^\infty} 
	  \|\Psi\|_{L^2}\|\Phi\|_{L^2}\\
	&\leq c_0'\| \Psi \|_{L^2} \| \Phi \|_{L^2}.
\end{align*}
Similarly, we have
\begin{align*}
	|D_H(\Psi, \Phi)| &=\left| \sum_{k=1}^N \int_\Omega \int_\Omega \frac{\sum_{j=1}^N(\psi_j(y)\co{\lambda_j(y)}+\co{\psi_j(y)}\lambda_j(y))}{|x-y|} \lambda_k(x) \co{\phi_k(x)} \dd y\dd x\right|\\
	&\leq \|\Lambda\|_{L^\infty}^2 \scalar{I_N\otimes\sum_{j=1}^N(|\co{\psi_j}|+|\psi_j|)\star w}{|\Phi|}\\
	&\leq\|\Lambda\|_{L^\infty}^2 \sqrt{N}\left\| \sum_{j=1}^N(|\co{\psi_j}|+|\psi_j|)\star w\right\|_{L^2(\Omega;\C)} \|\Phi\|_{L^2}\\
	&\leq 2N^\frac{3}{2} \left\|w\right\|_{L^1} \|\Lambda\|_{L^\infty}^2  \|\Psi\|_{L^2}\|\Phi\|_{L^2}\\
		&\leq c_0''\| \Psi \|_{L^2} \| \Phi \|_{L^2},
\end{align*}
where Young's inequality, see e.g. \cite[Theorem 14.6]{Schilling2005}, is used.
Together, we have the desired bound on $D$ with $c_0=c_0'+c_0''$.

For the second estimate, we first recall that  the embedding $H^1(0,T) \hookrightarrow C[0,T]$ is
continuous and compact (see, e.g.,  \cite{Ciarlet2013}),
hence there exists a positive constant $K$
such that $\| u \|_{C[0,T]} \leq K \| u \|_{H^1(0,T)}$ for any $u \in H^1(0,T)$; this is used for the control $u$.
Consequently, recalling \eqref{eq:bilinearForm}, we obtain the following estimate
\begin{equation}
\label{eq:pppp}
\begin{split}
  | B(\Psi,\Phi;u) |
&\leq \| \nabla \Psi \|_{L^2} \| \nabla \Phi \|_{L^2} +(1-\alpha)\bigl(|D(\Psi,\Phi)|+\|V(\Lambda)\Psi\|\|\Phi\|\bigr)\\
&\quad + \Bigl| \int_{\Omega} \scalarC{ V_0(x) \Psi(x) }{ \Phi(x) } \dd x \Bigr|
+ \Bigl| \int_{\Omega} \scalarC{ V_u(x) u(t) \Psi(x)}{ \Phi(x) } \dd x \Bigr| \\
&\leq \| \Psi \|_{H^1} \|\Phi \|_{H^1} +(c_0+\|V(\Lambda)\|_{L^\infty}) \| \Psi \|_{H^1} \|\Phi \|_{H^1}\\
&\quad + \| V_0 \|_{L^{\infty}} \| \Psi \|_{L^2} \| \Phi \|_{L^2}
+ K \| u \|_{H^1(0,T)} \| V_u \|_{L^{\infty}} \| \Psi \|_{L^2} \| \Phi \|_{L^2} \\
&= \Bigl( 1 + c_0+\|V(\Lambda)\|_{L^\infty}+\| V_0 \|_{L^{\infty}} + K \| u \|_{H^1(0,T)} \| V_u \|_{L^{\infty}} \Bigr) \| \Psi \|_{H^1} \| \Phi \|_{H^1}  ,
\end{split}
\end{equation}
hence there exists a constant $c_1$ such that \eqref{eq:est1} holds.

The estimate \eqref{eq:est1im} is easily verified with the above estimates for $D$, as $\scalar{\nabla \Psi}{\nabla \Psi}$ and $\scalar{V_{ext}\Psi}{\Psi}$ are real.

To prove the last statement, we recall \eqref{eq:bilinearForm},
and similar to \eqref{eq:pppp} we have the following
\begin{equation*}
\scalar{ \nabla \Psi }{ \nabla \Psi }
= B( \Psi , \Psi ; u ) - \scalar{ V_{ext} \Psi }{ \Psi } - (1-\alpha)D(\Psi,\Psi)-(1-\alpha)\scalar{V(\Lambda)\Psi}{\Psi}.\\
\end{equation*}
Taking the real part of this equation results in
\begin{equation}
\label{eq:pppp1}
\begin{split}
\scalar{ \nabla \Psi }{ \nabla \Psi } \leq \Re B( \Psi , \Psi ; u ) &+
\Bigl( \| V_0 \|_{L^{\infty}} + K \| u \|_{H^1(0,T)} \| V_u \|_{L^{\infty}}\\
 &+(1-\alpha)(c_0 +\|V(\Lambda)\|_{L^\infty})\Bigr) \| \Psi \|_{L^2}^2  .
\end{split}\end{equation}

Adding $\|\Psi\|_{L^2}^2$ on both sides we obtain
\begin{equation}
\label{eq:pppp3}
\| \Psi \|_{H^1}^2 \leq \Re B( \Psi , \Psi ; u ) +c_3 \| \Psi \|_{L^2}^2  ,
\end{equation}
where $c_3=\Bigl( \| V_0 \|_{L^{\infty}} + K \| u \|_{H^1(0,T)} \| V_u \|_{L^{\infty}} +c_0 +\|V(\Lambda)\|_{L^\infty}+1 \Bigr)$,
hence \eqref{eq:est2} holds.
\end{proof}

\section{A Galerkin approach}
\label{sec:Galerkin}
In this section, we introduce a finite-dimensional subspace $P_m$ of $H^1_0(\Omega; \C^N)$, and show existence of a unique solution of \eqref{eq:KSweak} in this subspace.
To this end, we take smooth functions ($C^\infty_0(\Omega)$)
$\phi_k=\phi_k(x)$, for $k=1,2,\dots$, such that
$\{\phi_k\}_k$ is an orthogonal basis for $H^1_0(\Omega)$ and
an orthonormal basis for $L^2(\Omega)$.
Further, we construct a basis $\{\Phi_k\}_k$ that is 
orthogonal for $H^1_0(\Omega;\C^N)$ and orthonormal for $L^2(\Omega;\C^N)$
by defining
\begin{equation}
\Phi_k(x) := \frac{1}{\sqrt{N}} \begin{pmatrix}
	\phi_{k}(x) \\ \vdots \\ \phi_{k}(x)\\ \vdots \\ \phi_{k}(x)
	\end{pmatrix}  ,
\end{equation}
where $\frac{1}{\sqrt{N}}$ is a normalization parameter.

For a fixed positive integer $m$, we define a function $\Psi_m$ as follows
\begin{align}
\label{eq:form}
\Psi_m(x,t) &:= \sum_{k=1}^m d^k_m(t) \Phi_k(x)\\
	&=\frac{1}{\sqrt{N}}\sum_{k=1}^{m} \begin{pmatrix}
	d_{m,1}^k(t)&&&&0\\
	&\ddots\\
	&& d_{m,j}^k(t)\\
	&&& \ddots\\
	0&&&& d_{m,N}^k(t)
	\end{pmatrix} \begin{pmatrix}
	 \phi_{k}(x)\\ \vdots \\ \phi_{k}(x)\\\vdots \\  \phi_{k}(x)
	\end{pmatrix}  ,
\end{align}
where the coefficients $d^k_{m,j}:[0,T] \rightarrow \R$ are such that
\begin{equation}
\label{eq:initial}
d^k_{m,j}(0) = \scalar{ \psi_{0,j} }{ \phi_{k,j} }  ,
\end{equation}
for $k= 1 , \dots , m$.
The space spanned by the first $m$ basis functions is called
\begin{equation*}
	P_m=\operatorname*{span}_{k=1,\dotsc,m}\{\Phi_k\}.
\end{equation*}

Moreover, by testing \eqref{eq:KSweak} for $\Phi=\Phi_k$, we obtain the following
\begin{equation}
\label{eq:KSweak_test}
i \scalar{ \partial_t \Psi_m }{ \Phi_k }
= B(\Psi_m,\Phi_k;u) + \alpha\scalar{ V(\Psi_m) \Psi_m }{ \Phi_k }
+ \scalar{ F }{ \Phi_k }  ,
\end{equation}
for almost all $0 \leq t \leq T$ and all $k =  1 , \dots , m$.
Thus we seek a solution $\Psi_m$ in the form \eqref{eq:form} that satisfies
the projection \eqref{eq:KSweak_test} of problem \eqref{eq:KSweak}
onto the finite dimensional subspace $W_m=L^2(0,T;P_m)$.

\begin{lemma}\label{lem:GlocallyLipschitz}
Recall that $\Psi_m := \sum_{k=1}^m d_m^k \Phi_k$ as in \eqref{eq:form} and define
$G^k: \C^N\rightarrow \C^N$ as $d_m\mapsto G^k(d_m):=\scalar{V(\Psi_m)\Psi_m}{\Phi_k}$, $d_m=(d_m^1,\dotsc,d_m^m)$.
The map $d_m \mapsto G^k(d_m)$ is locally Lipschitz continuous in $d_m$.
\end{lemma}
\begin{proof}
We want to show local Lipschitz continuity in $d_m$, i.e. that for every $\epsilon>0$
there exists a positive constant $L$ such that
$|G^k(d_m)-G^k(b_m)| \leq L |d_m-b_m^2|$ for all $d_m,b_m\in B_\epsilon(0)$.
	For a wave function in the Galerkin subspace with $d_m\in B_\epsilon(0)$, we have the following bounds
	\begin{align*}
		\|\Psi_m\|_{L^2}^2 &= \int_{\Omega} |\sum_{k=1}^m d_m^k \Phi_k(x)|^2\dd x\\
		&\leq m\sum_{k=1}^m |d_m^k|^2 \int_{\Omega} |\Phi_k(x)|^2\dd x
		=m\sum_{k=1}^m |d_m^k|^2\leq m^2 \epsilon^2,
	\end{align*}
	\begin{align*}
		\|\nabla\Psi_m\|_{L^2}^2 &= \int_{\Omega} |\sum_{k=1}^m d_m^k \nabla\Phi_k(x)|^2\dd x\\
		&\leq m\sum_{k=1}^m |d_m^k|^2 \int_{\Omega} |\nabla \Phi_k(x)|^2\dd x
		=m\sum_{k=1}^m |d_m^k|^2 C_m \leq m^2 \epsilon^2 C_m,
	\end{align*}
	with $C_m=\max_{k=1,\dotsc, m} \|\nabla \Phi_k\|_{L^2}^2$. From these two bounds, we obtain\linebreak $\|\Psi_m\|_{H^1}\leq (c_m+1) m^2\epsilon^2$.
	Now, we prove local Lipschitzianity for the different potentials. Consider $\Psi_m$ and $\Upsilon_m$ with coefficients in $B_\epsilon(0)$. We obtain the following
	\begin{equation}\label{eq:GalerkinVH}\begin{split}
		|\scalar{V_H(\Psi_m)\Psi_m-V_H(\Upsilon_m)\Upsilon_m}{\Phi_k}|
		&\leq \|V_H(\Psi_m)\Psi_m-V_H(\Upsilon_m)\Upsilon_m\|_{L^2} \|\Phi_k\|_{L^2}\\
		&\leq C_u \left( \|\Psi_m\|_{H^1}^2+\|\Upsilon_m\|_{H^1}^2\right)\|\Psi_m-\Upsilon\|_{L^2}\\
		&\leq L \|\Psi_m-\Upsilon_m\|_{L^2},
	\end{split}\end{equation}
	where the constant $L$ depends on the dimension of the Galerkin space $m$, the norm of the derivatives of the basis functions $C_m$ and $\epsilon$.
	
	For the exchange-correlation potential, we have from the Assumption \ref{assumptions} (\ref{assumptionVcBounded}) and (\ref{assumptionVxLipschitz}) the following estimates
	\begin{equation}\label{eq:GalerkinVxc}\begin{split}
		|\scalar{V_c(\Psi_m)\Psi_m-V_c(\Upsilon_m)\Upsilon_m}{\Phi_k}|
		&\leq \|V_c\|_{L^\infty}\|\Psi_m-\Upsilon_m\|_{L^2},\\
		|\scalar{V_x(\Psi_m)\Psi_m-V_x(\Upsilon_m)\Upsilon_m}{\Phi_k}|
		&\leq \tilde{L}\|\Psi_m-\Upsilon_m\|_{L^2}.
	\end{split}\end{equation}
	Using the estimates \eqref{eq:GalerkinVH} and \eqref{eq:GalerkinVxc}, we have
	\begin{align*}
		|G^k(d_m^1)-G^k(d_m^2)|
		&=|\scalar{V_H(\Psi_m)\Psi_m+V_x(\Psi_m)\Psi_m+V_c(\Psi_m)\Psi_m \right.\\
		&\left.\quad-V_H(\Upsilon_m)\Upsilon_m-V_x(\Upsilon_m)\Upsilon_m-V_c(\Upsilon_m)\Upsilon_m}{\Phi_k}|\\
		&\leq |\scalar{V_H(\Psi_m)\Psi_m-V_H(\Upsilon_m)\Upsilon_m}{\Phi_k}|\\
		&\quad+|\scalar{V_c(\Psi_m)+V_x(\Psi_m)-V_c(\Upsilon_m)-V_x(\Upsilon_m)}{\Phi_k}|\\
		&\leq L\|\Psi_m-\Upsilon_m\|_{L^2}+L'\|\Psi_m-\Upsilon_m\|_{L^2}.
	\end{align*}
	Further, we have
	\begin{align*}
		\|\Psi_m-\Upsilon_m\|_{L^2}^2&=\int_\Omega \left|\sum_{l=1}^m (d_{m,1}^l-d_{m,2}^l)\Phi_l(x)\right|^2 \dd x\\
		&\leq m \int_\Omega \sum_{l=1}^m |d_{m,1}^l-d_{m,2}^l|^2 |\Phi_l(x)|^2 \dd x\\
		&=m\sum_{l=1}^m |d_{m,1}^l-d_{m,2}^l|^2 \int_\Omega|\Phi_l(x)|^2 \dd x\\
		&\leq m |d_{m,1}(t)-d_{m,2}(t)|^2.
	\end{align*}
	
	All together, we have that $d_m \mapsto G^k(d_m)$ is locally Lipschitz continuous.
\end{proof}

To show existence of a unique solution in the finite-dimensional Galerkin space, we use the Carathéodory theorem, see, e.g., \cite{Walter1998}, because the time-dependent coefficients satisfy our differential equation only almost everywhere.
\begin{theorem}[Carathéodory]\label{thm:Caratheodory}
Consider the following initial value problem
\begin{equation}\label{eq:ODECaratheodory}
	\pa_t y(t)=f(t,y),\quad y(0)=\eta.
\end{equation}
Let $S=[0, T]\times \R^m$ and assume
	 $f$ satisfies $f(\cdot,y)\in L^1(0,T)$ for fixed $y$ and a generalized Lipschitz condition
	\begin{equation}
		|f(t,y)-f(t,\bar{y})|\leq l(t)|y-\bar{y}| \quad \text{in }S, \text{ where } l(t)\in L^1(0,T).
	\end{equation}
Then there exists a unique absolutely continuous solution satisfying \eqref{eq:ODECaratheodory} a.e. in $[0,T]$.
\end{theorem}

\begin{theorem}[Construction of approximate solutions]
\label{thm:approxS}
For each integer $m=1,2,\dotsc$ there exists a unique function $\Psi_m \in W_m$ of the form \eqref{eq:form} satisfying \eqref{eq:initial} and \eqref{eq:KSweak_test}.
\end{theorem}
\begin{proof}
Assuming $\Psi_m$ has the structure \eqref{eq:form}, we note from the fact that $\Phi_k$ are an orthonormal basis that
\begin{equation}
	\scalar{\pa_t\Psi_m(t)}{\Phi_k}=\pa_t{d_m^k}(t).
\end{equation}
Furthermore
\begin{align*}
	B(\Psi_m, \Phi_k; u)-(1-\alpha)D(\Psi_m, \Phi_k)&=\sum_{l=1}^{m} e^{kl}(t) d_m^l(t),\\
	D(\Psi_m, \Phi_k)&=\sum_{l=1}^m \tilde{e}^{kl} \Re(d_m^l(t)),
\end{align*}
for $e^{kl}:=B(\Phi_l,\Phi_k; u)-(1-\alpha)D(\Phi_l,\Phi_k)$, and $\tilde{e}^{kl}:= D(\Phi_l,\Phi_k)$ $k,l=1,\dotsc,m$. The real part comes from the definition of $D$ which already contains $\Re\Psi_m$.
Define $f^k(t):=\scalar{F(t)}{\Phi_k}$.

Then \eqref{eq:KSweak_test} becomes  a nonlinear system of ODEs as follows
\begin{align}\label{eq:ODEm}
	&i\pa_t{d_m^k}(t)=\sum_{l=1}^{m}e^{kl}(t) d_m^l(t)+(1-\alpha)\sum_{l=1}^{m}\tilde{e}^{kl}(t) \Re (d_m^l(t))+f^k(t)+\alpha G^k(d_m^l(t)),
\end{align}
for $k=1,\dotsc,m$ with the initial conditions \eqref{eq:initial}.

In \eqref{eq:ODEm}, the first term is linear, the second globally Lipschitz continuous with Lipschitz constant 1, and $f$ is constant with respect to $d_m$. By Lemma \ref{lem:GlocallyLipschitz}, $G^k$ is locally Lipschitz continuous in $d_m$ on every ball $B_r(0)$, so the right hand side is locally Lipschitz in $d_m$.
As $G^k(d_m(t))$ depends on $t$ only through $d_m^l(t)$ and $f\in L^2(0,T)$ and $e^{kl}\in H^1(0,T)$ through $u$, the right hand side is also in $L^2(0,T)$ and therefore the required $L^1(0,T)$-bound exists.
Hence, we can invoke the Carathéodory theorem  to show that \eqref{eq:ODEm} has a unique solution in the sense of Theorem \ref{thm:Caratheodory}.
\end{proof}

\section{Energy estimates}\label{sec:EnergyEstimates}
In this section, we discuss energy estimates concerning our evolution problem that are used to prove existence of solutions in $W$.
Further, we apply these energy estimates for solutions in $W$ to show uniqueness of the solution.

\begin{theorem}
\label{thm:estimates}
Let  $\Psi\in W_m$ be a solution of 
\begin{equation}\label{eq:WeakFormHilbertSpace}
i \scalar{ \partial_t \Psi(t) }{ \Phi }
= B(\Psi(t),\Phi;u(t)) + \alpha \scalar{ V(\Psi(t)) \Psi(t) }{ \Phi }
+ \scalar{ F(t) }{ \Phi }, \quad \forall \Phi \in P_m,
\end{equation}
a.e. in $(0,T)$. Then there exist positive constants $C$, $C_0$, $C_1$, $C'$ and $C''$
such that the following estimates hold
\begin{align}
\label{eq:estimate-1}
\max_{0\leq t \leq T} \| \Psi(t) \|_{L^2}^2 &\leq C \Bigl( \| \Psi_0 \|_{L^2}^2 + \| F \|_Y^2 \Bigr)  ,\\
\label{eq:Bestimate}
	\Re B(\Psi(t), \Psi(t); u)&\leq |B(\Psi(t), \Psi(t); u)|  \leq C_0(\|\Psi_0\|_{L^2}^2+\|F\|_{Y}^2) \text{ for a.a. }t\in [0,T],\\
\label{eq:H1estimate}
	\operatorname*{ess\,sup}_{0\leq t \leq T}\|\Psi(t)\|_{H^1}^2 &\leq C_1(\|\Psi_0\|_{L^2}^2+\|F\|_{Y}^2),\\
\label{eq:estimate-2}
	\|\Psi\|_X^2 &\leq C' \Bigl( \| F \|_Y^2 + \| \Psi_0 \|_{L^2}^2  +  \bigl( \| F \|_Y^2 + \| \Psi_0 \|_{L^2}^2 \bigr)^2 \Bigr) ,\\
\label{eq:estimate-3}
\| \Psi' \|_{X^*}^2 &\leq C'' \bigl(1+ \| F \|_{Y}^2 + \| \Psi_0 \|_{L^2}^2 \bigr)^3.
\end{align}
The same estimates hold for a $\Psi \in W$ solving \eqref{eq:KSweak}.
\end{theorem}
\begin{proof}
$\;$ \\
\underline{Estimate 1}\\
Testing \eqref{eq:WeakFormHilbertSpace} with $\Psi(\cdot,t)$, we obtain
\begin{equation}
\label{eq:KSweak_test2}
i \scalar{ \partial_t \Psi }{ \Psi }
= B(\Psi,\Psi;u) + \alpha \scalar{ V(\Psi) \Psi }{ \Psi }
+ \scalar{ F }{ \Psi }  ,
\end{equation}
a.e. in $(0,T)$.
This equation is equivalent to (see e.g. \cite{Evans2010})
\begin{equation}
\label{eq:KSweak_test3}
i \frac{1}{2} \abl{}{t} \| \Psi(t) \|_{L^2}^2
= B(\Psi,\Psi;u) + \alpha \scalar{ V(\Psi) \Psi }{ \Psi }
+ \scalar{ F }{ \Psi }  .
\end{equation}
Now, we notice that the left-hand side is purely imaginary,
while the terms \linebreak
$\scalar{ V(\Psi) \Psi }{ \Psi }$ and $B(\Psi,\Psi;u)$ apart from $D(\Psi, \Psi)$ 
are purely real.
Consequently, by splitting \eqref{eq:KSweak_test3} into
real and imaginary parts, we obtain the following
\begin{equation}
\label{eq:KSweak_real}
\frac{1}{2} \abl{}{t} \| \Psi(t) \|_{L^2}^2
= \Im \bigl( \scalar{ F }{ \Psi }+(1-\alpha)D(\Psi,\Psi) \bigr)  ,
\end{equation}
and
\begin{equation}
\label{eq:KSweak_imag}
\Re B(\Psi,\Psi;u) + \alpha \scalar{ V(\Psi) \Psi }{ \Psi }
+ \Re \bigl( \scalar{ F }{ \Psi } \bigr) = 0  .
\end{equation}

Now, using Lemma \ref{lem:bound-weak-form} and defining $\tilde{c}_0:=(1-\alpha)c_0$, equation \eqref{eq:KSweak_real} becomes as follows
\begin{equation}
\label{eq:KSweak_real2}
\begin{split}
\abl{}{t} \| \Psi(t) \|_{L^2}^2
&\leq 2 \| F \|_{L^2} \| \Psi \|_{L^2} +2\tilde{c}_0 \|\Psi\|_{L^2}^2\\
&\leq \| F \|_{L^2}^2 +  (1+2\tilde{c}_0) \| \Psi \|_{L^2}^2  .
\end{split}
\end{equation}
By defining $\eta(t) := \| \Psi(t) \|_{L^2}^2$
and $\xi(t) := \| F(t) \|_{L^2}^2$ the previous inequality becomes as follows
\begin{equation}
\eta'(t) \leq (1+2\tilde{c}_0)\eta(t) + \xi(t)  ,
\end{equation}
a.e. in $(0,T)$. Thus, by applying the Gronwall inequality \cite{Evans2010} in the differential
form, we obtain the following
\begin{equation}
\label{eq:afterGronw}
\eta(t) \leq e^{(1+2\tilde{c}_0)t} \Bigl( \eta(0) + \int_0^t \xi(s) \dd s \Bigr)  .
\end{equation}
Notice that by \eqref{eq:initial}, it holds that
$\eta(0) =  \| \Psi(0) \|_{L^2}^2 = \| \Psi_0 \|_{L^2}^2$.
Consequently, by using \eqref{eq:afterGronw}, we know that there exists
a positive constant $C$ such that the following estimate holds
\begin{equation}
\label{eq:estimate-A}
\max_{0\leq t \leq T} \| \Psi(t) \|_{L^2}^2
\leq C \Bigl( \| \Psi_0 \|_{L^2}^2 + \| F \|_Y^2 \Bigr)  .
\end{equation}

For $\Psi\in W$, we have the continuous embedding $W \hookrightarrow C([0,T];L^2(\Omega; \C^N))$, see, e.g. \cite[p. 287]{Evans2010}.
 With this, we can evaluate $\Psi$ at time $t$ and find the same estimate if $\Psi\in W$ solves \eqref{eq:KSweak}.

$ $

\underline{Estimate 2}\\
	Taking the real part of \eqref{eq:KSweak_test3}, we find
	\begin{align*}
		\Re B(\Psi, \Psi; u)+\alpha \scalar{V(\Psi)\Psi}{\Psi}+\Re\scalar{F}{\Psi}=0.
	\end{align*}
	Using that $V_H\geq 0$ and $V_c\in L^\infty(\Omega)$, we get
	\begin{align*}	
		\Re B(\Psi, \Psi; u)&=\alpha\bigl(-\scalar{V_H(\Psi)\Psi}{\Psi}-\scalar{V_x(\Psi)\Psi}{\Psi}-\scalar{V_c(\Psi)\Psi}{\Psi}\bigr)-\Re\scalar{F}{\Psi}\\
		&\leq |\scalar{V_x(\Psi)\Psi}{\Psi}|+|\scalar{V_c(\Psi)\Psi}{\Psi}|+|\Re\scalar{F}{\Psi}|\\
		&\leq |\scalar{V_x(\Psi)\Psi}{\Psi}|+C_{V_c}\|\Psi\|_{L^2}^2+ \|F\|_{L^2}^2+\|\Psi\|_{L^2}^2.
	\end{align*}
	From the assumption \ref{assumptions} (\ref{assumptionVxLipschitz})
	and using \eqref{eq:estimate-1}, we obtain the following
	\begin{align*}
		\Re B(\Psi, \Psi; u)\leq C  \|\Psi\|_{L^2}^2+\|F\|_{L^2}^2
		\leq C_0' (\|\Psi_0\|_{L^2}^2+\|F\|_{Y}^2).
	\end{align*}
	By Lemma \ref{lem:bound-weak-form}, it holds that $\Im B(\Psi, \Psi; u)\leq c_0 \|\Psi\|_{L^2}^2$. Combining these two estimates one concludes \eqref{eq:Bestimate}.
	
	As for the first estimate, the same applies in the case when $\Psi\in W$ solves \eqref{eq:KSweak}.
	
$ $
	
\underline{Estimate 3}\\
	For the second bound, we simply combine Lemma \ref{lem:bound-weak-form} with \eqref{eq:Bestimate} and \eqref{eq:estimate-1}. We have
	\begin{align*}
		\|\Psi(t)\|_{H^1}^2 &\leq \Re B(\Psi(t), \Psi(t); u(t))+c_3\|\Psi(t)\|_{L^2}^2\\ 
		&{\leq}C_0(\|\Psi_0\|_{L^2}^2+\|F\|_{Y}^2)+c_3\|\Psi(t)\|_{L^2}^2\\
		&{\leq}(C_0+c_3C)(\|\Psi_0\|_{L^2}^2+\|F\|_Y^2). 
	\end{align*}
	
	If $\Psi \in W$, one has to use the fact that given $u_k\rightharpoonup u$ in $L^2(0,T; H^1_0(\Omega))$ with the uniform bound $\operatorname*{ess\,sup}_{0\leq t \leq T} \|u_k(t)\|\leq C$ it follows that $\operatorname*{ess\,sup}_{0\leq t \leq T} \|u(t)\|\leq C$. With this fact, we obtain
	\begin{align*}
	\operatorname*{ess\,sup}_{0\leq t \leq T}\|\Psi(t)\|_{H^1}^2 
			&\leq C_0(\|\Psi_0\|_{L^2}^2+\|F\|_{Y}^2)+c_3C(\|\Psi_0\|_{L^2}^2+\|F\|_Y^2).
	\end{align*}
	
$ $
	
\underline{Estimate 4}\\
First, we need an adequate bound for the term $\scalar{ V(\Psi) \Psi }{ \Phi }$
for any $\Phi \in L^2(\Omega;\C^N)$.
For this reason,
we write the following
\begin{equation}
\label{eq:bound_pot}
\begin{split}
\scalar{ V(\Psi) \Psi }{ \Psi }
=\scalar{ V_H(\Psi) \Psi }{ \Psi }+\scalar{ V_x(\Psi) \Psi }{ \Psi }+\scalar{ V_c(\Psi) \Psi }{ \Psi }.
\end{split}
\end{equation}
To bound $V_H$, we use the Cauchy-Schwarz inequality, Lemma \ref{CancesLemma}, and \eqref{eq:H1estimate} to arrive at
\begin{align}\label{eq:bound_potVH}
	\scalar{V_H(\Psi)\Psi}{\Psi} &\leq C_u\|\Psi\|_{H^1}^2 \|\Psi\|_{L^2}^2
	\leq C_1(\|\Psi_0\|_{L^2}^2+\|F\|_{Y}^2)\|\Psi\|_{L^2}^2.
\end{align}
Next, we recall that $x \mapsto V_c(x,\cdot)$ is bounded (Assumption \ref{assumptions} (\ref{assumptionVcBounded})) and $V_x$ is Lipschitz continuous (Assumption \ref{assumptions} (\ref{assumptionVxLipschitz})).
Consequently, from \eqref{eq:bound_pot}, it follows that there exists
a positive constant $K'$ such that the following holds
\begin{equation}
\label{eq:bound_pot2}
\scalar{ V(\Psi) \Psi }{ \Psi } \leq K' \| \Psi \|_{L^2}^2 ,
\end{equation}
where $K'$ depends on $\|\Psi_0\|_L^2$ and $\|F\|_Y$.

By summing term-by-term \eqref{eq:KSweak_real} with \eqref{eq:KSweak_imag},
we get the following
\begin{equation}
\label{eq:KSweak_new}
\begin{split}
\frac{1}{2} \abl{}{t} \| \Psi \|_{L^2}^2 + \Re B(\Psi,\Psi;u)
= &\Im  \scalar{ F }{ \Psi } +(1-\alpha)\Im D(\Psi,\Psi) \\
&- \alpha\scalar{ V(\Psi) \Psi }{ \Psi }- \Re  \scalar{ F }{ \Psi } .
\end{split}
\end{equation}
Adding to both sides the term $c_3 \| \Psi \|_{L^2}^2$,
where $c_3$ is the same as in Lemma \ref{lem:bound-weak-form},
we obtain the following
\begin{equation}
\label{eq:KSweak_new2}
\begin{split}
&\frac{1}{2} \abl{}{t} \| \Psi(t) \|_{L^2}^2 + \Re B(\Psi,\Psi;u)
+ c_3 \| \Psi \|_{L^2}^2\\
&= \Im \scalar{ F }{ \Psi } +(1-\alpha)\Im D(\Psi,\Psi)
+ c_3 \| \Psi \|_{L^2}^2
- \alpha\scalar{ V(\Psi) \Psi }{ \Psi }
- \Re \scalar{ F }{ \Psi }  .
\end{split}
\end{equation}
Next, by applying Lemma \ref{lem:bound-weak-form} and using \eqref{eq:bound_pot2}
we get the following
\begin{equation}
\label{eq:KSweak_new3}
\frac{1}{2} \abl{}{t} \| \Psi(t) \|_{L^2}^2 + 
 \| \Psi \|_{H^1}^2
\leq \| F \|_{L^2}^2 + ( 1 + c_0 + c_3 + K' ) \| \Psi \|_{L^2}^2  .
\end{equation}
By manipulating \eqref{eq:KSweak_new3} and integrating over $(0,T)$, we have
\begin{equation}
\label{eq:KSweak_new4}
\int_0^T \| \Psi \|_{H^1}^2 \dd t
\leq \int_0^T \| F \|_{L^2}^2 + ( 1 + c_0 + c_3 + K' ) \| \Psi \|_{L^2}^2 \dd t
- \int_0^T \frac{1}{2} \abl{}{t} \| \Psi(t) \|_{L^2}^2 \dd t  ,
\end{equation}
which implies that
\begin{equation}
\label{eq:KSweak_new5}
\begin{split}
 \| \Psi \|_X^2
&\leq \| F \|_Y^2
+ ( 1 + c_0 + c_3 + K' ) C T\Bigl(\| \Psi_0 \|_{L^2}^2 + \| F \|_Y^2 \Bigr)
+ \frac{1}{2} \Bigl( \| \Psi_0 \|_{L^2}^2 - \| \Psi(T) \|_{L^2}^2 \Bigr) \\
&\leq \| F \|_Y^2
+ ( 1 + c_0 + c_3 + K' ) C T\Bigl(  \| \Psi_0 \|_{L^2}^2 + \| F \|_Y^2 \Bigr)
+ \frac{1}{2} \| \Psi_0 \|_{L^2}^2 \\
&= \bigl(1+( 1 + c_0 + c_3 + K' ) TC\bigr) \| F \|_Y^2
+ \left( ( 1 + c_0 + c_3 + K' ) C T + \frac{1}{2} \right) \| \Psi_0 \|_{L^2}^2  ,
\end{split}
\end{equation}
where  we used \eqref{eq:estimate-1}.
Using the dependence of $K'$ on the data, the previous \eqref{eq:KSweak_new5} implies that there exists a
positive constant $C'$ such that
\begin{equation}\label{eq:KSweak_new6}
	\|\Psi\|_X^2\leq C' \Bigl( \| F \|_Y^2 + \| \Psi_0 \|_{L^2}^2  +  \bigl( \| F \|_Y^2 + \| \Psi_0 \|_{L^2}^2 \bigr)^2 \Bigr).
\end{equation}
The same calculation can be done for $\Psi\in W$ being a solution of \eqref{eq:KSweak}.
$\;$ \\

\underline{Estimate 5}\\
Fix any $v \in H^1_0(\Omega;\C^N)$, with $\| v \|_{H^1} \leq 1$.
Write $v = v_1 + v_2$, where $v_1 \in \Span\{\phi_k\}_{k=1}^{m}$
and $\scalar{v_2}{\phi_k}=0$ for $k=1,\dots,m$.
Since the functions $\{\phi_k\}_{k=1}^{\infty}$ are orthogonal
in $H^1_0(\Omega)$, we have
\begin{equation}
\label{eq:norm-leq-1}
1 \ge \| v \|_{H^1}^2 = \scalarHOne{v_1+v_2}{v_1+v_2}
= \| v_1 \|_{H^1}^2 + \| v_2 \|_{H^1}^2 \geq \| v_1 \|_{H^1}^2 .
\end{equation}
Next, utilizing \eqref{eq:KSweak} with $\Psi\in W_m$, we obtain
\begin{equation}
\label{eq:KSweak_test_boo}
i \scalar{ \partial_t \Psi }{ v_1 }
= B(\Psi,v_1;u) + \alpha\scalar{ V(\Psi) \Psi }{ v_1 }
+ \scalar{ F }{ v_1 } 
\end{equation}
a.e. in $[0, T]$. Using the decomposition of $v$, this implies that
\begin{equation}\label{eq:Hm1normsplit}
\begin{split}
|\langle \Psi' , v \rangle | &= |\scalar{ \partial_t \Psi }{ v }| \\
&= |\scalar{ \partial_t \Psi }{ v_1 }| \\
&= | B(\Psi,v_1;u) + \alpha\scalar{ V(\Psi) \Psi }{ v_1 }
+ \scalar{ F }{ v_1 } |  ,
\end{split}
\end{equation}
where $\partial_t \Psi \in L^2((0,T);H^1_0(\Omega))$ is the Riesz representative
of $\Psi' \in L^2((0,T);H^{-1}(\Omega))$ and 
$\dualP{ \cdot }{ \cdot } : H^{-1}(\Omega) \times H^{1}_0(\Omega) \rightarrow \C$ denotes the dual pairing for $H^1_0(\Omega)$ and its dual $H^{-1}(\Omega)$.

By using the Cauchy-Schwarz inequality and Assumptions \ref{assumptions} (\ref{assumptionVcBounded}) and (\ref{assumptionVxLipschitz}) and $\|v_1\|_{H^1}\leq 1$, we have that there exists a positive constant $\tilde{K}$ such that
\begin{align*}
	|\scalar{ (V_x(\Psi)+V_c(\Psi)) \Psi }{ v_1 }| \leq \|(V_x(\Psi)+V_c(\Psi)) \Psi\|_{L^2}\|v_1\|_{L^2}\leq \tilde{K}\|\Psi\|_{L^2}.
\end{align*}

By recalling Lemma \ref{lem:bound-weak-form}, \eqref{eq:bound_potVH} and $\| v \|_{H^1} \leq 1$, 
we obtain that there exists a positive constant $\tilde{C}$ such that
\begin{equation}
\label{eq:booo}
|\langle \Psi' , v \rangle |
\leq \tilde{C}(1+\|\Psi_0\|_{L^2}^2+\|F\|_Y^2) \bigl( \| F \|_{L^2} + \| \Psi \|_{H^1} \bigr)  ,
\end{equation}
and from \eqref{eq:booo}, we have the following
\begin{equation}
\label{eq:booo2}
\| \Psi' \|_{H^{-1}}=\sup_{0\neq v\in H_0^1(\Omega)}\frac{|\langle \Psi', v\rangle|}{\|v\|_{H^1}}
\leq \tilde{C}(1+\|\Psi_0\|_{L^2}^2+\|F\|_Y^2) \bigl( \| F \|_{L^2} + \| \Psi \|_{H^1} \bigr)  .
\end{equation}
This implies that
\begin{align*}
	\|\Psi'\|_{H^{-1}}^2 \leq \tilde{C}^2(1+\|\Psi_0\|_{L^2}^2+\|F\|_Y^2)^2 \bigl( \| F \|_{L^2} + \| \Psi \|_{H^1} \bigr)^2\\
	\leq 2\tilde{C}^2(1+\|\Psi_0\|_{L^2}^2+\|F\|_Y^2)^2\bigl( \| F \|_{L^2}^2 + \| \Psi \|_{H^1}^2 \bigr).
\end{align*}
By integrating over $(0,T)$ and using \eqref{eq:estimate-2},
we obtain that there exists a positive constant $C''$ such that
the following estimate holds
\begin{equation}
\| \Psi' \|_{X^*}^2
\leq C'' \bigl(1+ \| F \|_{Y}^2 + \| \Psi_0 \|_{L^2}^2 \bigr)^3  ,
\end{equation}
where $X^* = L^2((0,T);H^{-1}(\Omega))$ and the proof for $\Psi\in W_m$ is completed.

For $\Psi\in W$, no decomposition is necessary in \eqref{eq:Hm1normsplit}, so we can use $v_1=v, v_2=0$ and apply the same estimates to conclude our proof.
\end{proof}

\section{Existence of a weak solution}\label{sec:ExistenceSolution}
In the preceding section, we have shown the estimates in Theorem \ref{thm:estimates} for solutions  $\Psi_m\in W_m$ in the Galerkin subspace. In this section, we use these estimates to show the existence of a solution in the full Sobolev space $W$. To this end, we make use of the following embedding theorem by Lions \cite[1.5.2]{Lions1969}.
\begin{lemma}
	\label{lem:Lions}
	Given three Banach spaces $B_0 \Subset B \hookrightarrow B_1$ with $B_0$, $B_1$ reflexive and the embedding $B_0 \hookrightarrow B$ being compact,
	then the space
	\begin{align*}
		V&=\left\{v| v\in L^p((0,T), B_0), v'\in L^q((0,T), B_1) \right\}, \quad 1<p,q<\infty,\\
		\|v\|_V&:=\|v\|_{L^p((0,T), B_0)}  +\|v'\|_{L^q((0,T), B_1)}
	\end{align*}
	is compactly embedded in $L^p(0,T; B)$.
\end{lemma}

\begin{theorem}
Problem \eqref{eq:KSweak} admits a weak solution, i.e. there exists a $\Psi\in W$ such that
\begin{equation}
\begin{split}
&i \scalar{ \partial_t \Psi}{ \Phi }
= B(\Psi,\Phi;u) + \alpha\scalar{ V(\rho) \Psi}{ \Phi } + \scalar{ F }{ \Phi } \\
&\text{a.e. in } (0,T), \; \forall \Phi \in H^1_0(\Omega;\C^N),\\
&\Psi_0\in L^2(\Omega;\C^N) .
\end{split}
\end{equation} 
\end{theorem}
\begin{proof}
Consider a sequence $\{\Psi_m\}_{m=1}^{\infty}$ of solutions of the Galerkin problem \eqref{eq:WeakFormHilbertSpace}, then according to the estimates
\eqref{eq:estimate-1}, \eqref{eq:estimate-2}, and \eqref{eq:estimate-3}
in Theorem \ref{thm:estimates}, the sequence
 is bounded in $X$ and 
$\{\Psi_m'\}_{m=1}^{\infty}$ is bounded in $X^*$.
Consequently, there exists a subsequence $\{ \Psi_{m_l} \}_{l=1}^{\infty}$
and a function $\Psi \in X$ with $\Psi' \in X^*$ such that
$\Psi_{m_l} \rightharpoonup \Psi$ in $X$ and
$\Psi_{m_l}' \rightharpoonup \Psi'$ in $X^*$; see, e.g., \cite{Evans2010}.
Moreover, by Lions' theorem (Lemma \ref{lem:Lions}) we know that $W$
is compactly embedded in $Y:=L^2(0,T;L^2(\Omega))$,
consequently, we have strong convergence of the subsequence
$\Psi_{m_l} \rightarrow \Psi$ in $Y$.

Next, we fix a positive integer $M$ and construct a test function\linebreak
$\Phi \in C^1([0,T];H^1_0(\Omega;\C^N))$ as follows
\begin{equation}
\label{eq:form2}
\Phi(x,t) := \sum_{k=1}^M d^k_m(t) \Phi_k(x)  ,
\end{equation}
where $\{ d^k_m \}_{k=1}^{M}$ are given smooth functions.
We choose $m \geq M$, multiply \eqref{eq:KSweak_test} by $d^k_m(t)$,
sum over $k=1,\dots,M$,
 and integrate with respect to $t$
to obtain the following
\begin{equation}
\label{eq:KSweak_exist}
\int_0^T i \dualP{ \Psi_m' }{ \Phi } \dd t
= \int_0^T B(\Psi_m,\Phi;u) + \alpha\scalar{ V(\Psi_m) \Psi_m }{ \Phi }
+ \scalar{ F }{ \Phi } \dd t.
\end{equation}

By setting now $m=m_l$ and by recalling continuity of $V(\Psi)$ from Lemma \ref{lem:VContinuousFunction}
 and
strong convergence $\Psi_{m_l} \rightarrow \Psi$ in $Y$,
we can pass to the limit to obtain
\begin{equation}
\label{eq:KSweak_exist2}
\int_0^T i \dualP{ \Psi' }{ \Phi } \dd t
= \int_0^T B(\Psi,\Phi;u) + \alpha\scalar{ V(\Psi) \Psi }{ \Phi }
+ \scalar{ F }{ \Phi } \dd t  .
\end{equation}
This equality holds for all $\Phi \in X$ as functions of the form \eqref{eq:form2}
are dense in $X$.
Hence, in particular
\begin{equation}
\label{eq:KSweak_exist3}
i \dualP{ \Psi' }{ v }
= B(\Psi,\Phi;u) + \alpha\scalar{ V(\Psi) \Psi }{ v }
+ \scalar{ F }{ v }  ,
\end{equation}
for any $v \in H^1_0(\Omega;\C^N)$ and a.e. in $[0, T]$.
From \cite[Theorem 3 p. 287]{Evans2010}, we know also that
$\Psi \in C([0,T];L^2(\Omega;\C^N))$.

It remains to prove that $\Psi(\cdot,0) = \Psi_0$.
For this purpose, we first notice from \eqref{eq:KSweak_exist2} that the following holds
\begin{equation}
\label{eq:KSweak_exist4}
\int_0^T - i \dualP{ \Phi' }{ \Psi } \dd t
= \int_0^T B(\Psi,\Phi;u) + \alpha\scalar{ V(\Psi) \Psi }{ \Phi }
+ \scalar{ F }{ \Phi } \dd t + \scalar{ \Psi(0) }{ \Phi(0) }  ,
\end{equation}
for any $\Phi \in C^1([0,T];H^1_0(\Omega;\C^N))$ with $\Phi(T)=0$.
Similarly, from \eqref{eq:KSweak_exist} we get
\begin{equation}\label{eq:KSweak_exist5}\begin{split}
\int_0^T - i \dualP{ \Phi' }{ \Psi_m } \dd t
&= \int_0^T B(\Psi_m,\Phi;u) + \alpha\scalar{ V(\Psi_m) \Psi_m }{ \Phi }
+ \scalar{ F }{ \Phi } \dd t\\
&\quad + \scalar{ \Psi_m(0) }{ \Phi(0) }  .
\end{split}\end{equation}
We set $m=m_l$ and use again the considered convergences to find
\begin{equation}
\label{eq:KSweak_exist6}
\int_0^T - i \dualP{ \Phi' }{ \Psi } \dd t
= \int_0^T B(\Psi,\Phi;u) + \alpha\scalar{ V(\Psi) \Psi }{ \Phi }
+ \scalar{ F }{ \Phi } \dd t + \scalar{ \Psi_0 }{ \Phi(0) }  ,
\end{equation}
because $\Psi_{m_l}(0) \rightarrow \Psi_0$.
As $\Phi(0)$ is arbitrary, by comparing \eqref{eq:KSweak_exist4}
and \eqref{eq:KSweak_exist6} we conclude that $\Psi(0) = \Psi_0$.
\end{proof}

\section{Uniqueness of a weak solution}
We have shown that there exists at least one solution $\Psi \in W$ of \eqref{eq:KSweak}. Now, we can apply the extension of Theorem \ref{thm:estimates} to the space $W$ and use the Lipschitz properties of the potentials to show that the solution is indeed unique.

\label{sec:Uniqueness}
\begin{theorem}
The weak form of the Kohn-Sham equations \eqref{eq:KSweak} is uniquely solvable.
\end{theorem}
\begin{proof}
Seeking a contradiction, we assume that there exists two distinct weak solutions
of \eqref{eq:KSweak} $\Psi$ and $\Upsilon$ in $W$ with $\| \Psi - \Upsilon \|_X > 0$.
Therefore, we have
\begin{equation}
\label{eq:KSweakUniq1}
\begin{split}
&i \scalar{ \partial_t \Psi }{ \Phi }
= B(\Psi,\Phi;u) + \alpha\scalar{ V(\Psi) \Psi}{ \Phi } + \scalar{ F }{ \Phi }  ,
\end{split}
\end{equation}
and
\begin{equation}
\label{eq:KSweakUniq2}
\begin{split}
&i \scalar{ \partial_t \Upsilon }{ \Phi }
= B(\Upsilon ,\Phi;u) + \alpha\scalar{ V(\Upsilon ) \Upsilon }{ \Phi } + \scalar{ F }{ \Phi }  ,
\end{split}
\end{equation}
for all test functions $\Phi \in H^1_0(\Omega; \C^N)$.
Subtracting term-by-term \eqref{eq:KSweakUniq2} from \eqref{eq:KSweakUniq1}
and defining $\hat{\Psi} := \Psi - \Upsilon$
we obtain the following
\begin{equation}
\label{eq:KSweakUniq3}
\begin{split}
i \scalar{ \partial_t \hat{\Psi} }{ \Phi }
&= B(\hat{\Psi} ,\Phi;u)
+ \alpha\scalar{ V(\Psi ) \Psi - V(\Upsilon ) \Upsilon }{ \Phi }  .
\end{split}
\end{equation}
By testing the previous \eqref{eq:KSweakUniq3} with $\Phi = \hat{\Psi}(t)$,
we obtain 
\begin{equation}
\label{eq:KSweakUniq4}
i \scalar{ \partial_t \hat{\Psi} }{ \hat{\Psi} }
= B(\hat{\Psi},\hat{\Psi};u)
+ \alpha\scalar{ V(\Psi ) \Psi - V(\Upsilon ) \Upsilon }{ \hat{\Psi} }  .
\end{equation}
Similarly, as for \eqref{eq:KSweak_test3} we have
\begin{equation}
\label{eq:KSweak_test4}
i \frac{1}{2} \abl{}{t} \| \hat{\Psi} \|_{L^2}^2
= B(\hat{\Psi},\hat{\Psi};u)
+ \alpha\scalar{ V(\Psi ) \Psi - V(\Upsilon ) \Upsilon }{ \hat{\Psi} }  .
\end{equation}
Now, we notice that the left-hand side is purely imaginary.
Consequently, by taking the imaginary part of \eqref{eq:KSweak_test4},
we obtain the following
\begin{equation}
\label{eq:KSweakU-imag}
\frac{1}{2} \abl{}{t} \| \hat{\Psi} \|_{L^2}^2
=\alpha\Im \left( \scalar{ V(\Psi ) \Psi - V(\Upsilon ) \Upsilon }{ \hat{\Psi} } +(1-\alpha)D(\hat{\Psi},\hat{\Psi})\right).
\end{equation}
From \eqref{eq:KSweakU-imag} and \eqref{eq:lem3:estD} in  Lemma \ref{lem:bound-weak-form}, we get
\begin{equation*}
\begin{split}
\abl{}{t} \| \hat{\Psi} \|_{L^2}^2
&= 2\alpha\Im \bigl( \scalar{ V(\Psi ) \Psi - V(\Upsilon ) \Upsilon }{ \hat{\Psi} } \bigr)+2(1-\alpha)\Im D(\hat{\Psi},\hat{\Psi})  \\
&\leq  \| V(\Psi) \Psi - V(\Upsilon ) \Upsilon  \|_{L^2} \| \hat{\Psi} \|_{L^2} + 2c_0 \| \hat{\Psi} \|_{L^2}^2.\\
\intertext{Using Lemma \ref{CancesLemma}, Assumptions \ref{assumptions} (\ref{assumptionVcBounded}), (\ref{assumptionVxLipschitz}), \eqref{eq:estimate-1}, and Theorem \ref{thm:estimates}, we obtain}
\abl{}{t} \| \hat{\Psi} \|_{L^2}^2&\leq C_u'(\|\Psi\|_{H^1}^2+\|\Upsilon\|_{H^1}^2) \|\hat{\Psi}\|_{L^2}^2+2c_0\| \hat{\Psi} \|_{L^2}^2\\
&\leq  c^\#\bigl(L+K+\| F \|_{L^2}^2 + \| \Psi_0 \|_{L^2}^2+c_0 \bigr) \| \hat{\Psi} \|_{L^2}^2.
\end{split}
\end{equation*}
By defining $\eta(t) := \| \hat{\Psi} \|_{L^2} ^2$ and
$\vartheta(t) := c^\#\bigl(L+K+\| F \|_{L^2}^2 + \| \Psi_0 \|_{L^2}^2+c_0 \bigr)$,
we obtain the following inequality
\begin{equation}
\eta'(t) \leq \vartheta(t) \eta(t)  .
\end{equation}
By applying the Gronwall's inequality, we obtain the following
\begin{equation}
\eta(t) \leq \exp \Bigl( \int_0^t \vartheta(s) ds \Bigr) \eta(0)  .
\end{equation}
By noticing that
\begin{equation}
\begin{split}
\int_0^t \vartheta(s) ds &= \int_0^t c^\#\bigl(L+K+\| F \|_{L^2}^2 + \| \Psi_0 \|_{L^2}^2+c_0 \bigr) \dd s \\
&\leq \int_0^T c^\#\bigl(L+K+\| F \|_{L^2}^2 + \| \Psi_0 \|_{L^2}^2+c_0 \bigr) \dd s \\
&\leq c^\#\bigl( \| F \|_{Y}^2 + T \| \Psi_0 \|_{L^2}^2 +T (c_0+L+K) \bigr)  ,
\end{split}
\end{equation}
and by recalling that $\eta(0) = \| \hat{\Psi}(0) \|_{L^2} = 0$, we obtain that
$\| \hat{\Psi} \|_{L^2} \leq 0$ a.e. in $(0,T)$, and the claim follows.
\end{proof}

\section{Improved regularity}\label{sec:ImprovedRegularity}

We have established the existence and uniqueness of a solution to \eqref{eq:KSweak} in $W$. Although our methodology and assumptions on $V$ differ from \cite{Jerome2015}, our result is similar to \cite{Jerome2015}.
Now, we improve these results in the case $\Psi_0\in H_0^1(\Omega)$. With this setting, we prove that the solution to \eqref{eq:KSweak} is twice weakly differentiable in space and its first spatial derivative is bounded.

\begin{lemma}[Difference quotients]\label{lem;diffquot} Assume that for fixed $u\in L^p(V)$, $1<p<\infty$, $V\subset\subset \Omega$, there exists a constant $C$ such that $\|D^h u\|_{L^p(V)} \leq C$ for all \linebreak $0<|h|<\frac{1}{2} \operatorname{dist}(V, \pa\Omega)$ where
\begin{align*}
	D_i^hu(x)=\frac{u(x+e_i h)-u(x)}{h},\quad D^hu=(D_1^hu,\dotsc , D_n^hu).
\end{align*}
Then
\begin{align*}
	u\in H^{1,p}(V), \text{ with } \|Du\|_{L^p(V)} \leq C,
\end{align*}
where $C$  may depend on $u$, e.g. on $\|u\|_{L^p}(\Omega)$. Furthermore, the statement holds for the case of two half-balls $\Omega=\{|x|<R\}\cap \{x_n> 0\}$ and $V=\{|x|<\frac{R}{2}\}\cap \{x_n> 0\}$.
\end{lemma}
\begin{proof}
	See \cite[§5.8.2, Theorem 3]{Evans2010} and the remark after the proof.
\end{proof}

Next, we extend the result in \cite[§6.3.2, Theorem 4]{Evans2010} for linear elliptic problems to the case of a specific nonlinear problem.

\begin{lemma}\label{lem:EllipticRegularity}
	Let $\varphi\in H_0^1(\Omega; \C^N)$ be a weak solution of the elliptic boundary value problem
	\begin{align*}
		B(\varphi,v;u)+\scalar{V(\varphi)\varphi}{v}=\scalar{A}{v}, \forall v\in H_0^1(\Omega; \C^N), \; A\in L^2(\Omega; \C^N),
	\end{align*}
	such that $\|\varphi\|_{L^2}^2\leq \gamma \|A\|_{L^2}^2$ holds. Furthermore be $\partial \Omega\in C^2$. Then $\varphi\in H^2(\Omega; \C^N)$ and
	\begin{align*}
		\|\varphi\|_{H^2} &\leq c\left(\|A\|_{L^2}+\|\varphi\|_{L^2} \right),\\
		\text{where }c&=\max\left\{1,\|V_0\|_{L^\infty}+\|u\|_{C(0,T)}\|V_u\|_{L^\infty}+\|\varphi\|_{H^1}^2+L+K\right\}.
	\end{align*}
\end{lemma}
\begin{proof}
	To extend the results in \cite[§6.3.2, Theorem 4]{Evans2010}, two issues have to be treated carefully. First, the nonlinear potential has to be bounded in a suitable way and, second, extra care has to be taken when changing the coordinates.

	The nonlinear potential has to be bounded in such a way that Lemma \ref{lem;diffquot} can be applied. Therefore, we need to find a constant $c$ such that $\|V(\varphi)\varphi\|_{L^2}^2\leq c\|\varphi\|_{L^2}^2$, where $c$ is allowed to depend on $\varphi$.
	This can be done using Lemma \ref{CancesLemma} as follows
	\begin{align*}
		\|V_H(\varphi)\varphi\|_{L^2} \leq C_u\|\varphi\|_{H^1}^2 \|\varphi\|_{L^2},
	\end{align*}	
	and using Assumptions \ref{assumptions} (\ref{assumptionVxLipschitz}) and (\ref{assumptionVcBounded}), we obtain
	\begin{align*}
		\|V(\varphi)\varphi\|_{L^2}^2 \leq (\|\varphi\|_{H^1}^4+L^2+K^2)\|\varphi\|_{L^2}^2.
	\end{align*}
	Now, we can apply Lemma \ref{lem;diffquot} to obtain that the solution is in $H^2(U)$ for a half-ball $U$.
	
	Furthermore, in the proof it is necessary to locally flatten out the boundary. This is done by a $C^2$-map that keeps all the coordinates apart from one dimension which is transformed onto a line. This ensures that the determinant of the Jacobian is equal to one.
	
	The coordinate transformation of the Laplacian and the linear external potential is as for standard parabolic PDEs. The exchange and correlation potentials do not explicitly depend on space and time but only pointwise on the wave function. Hence a change of coordinates does not change the potential. For the Hartree potential, however, more care is needed. Let the change of coordinates be given by
	\begin{align*}
		x&=k(\hat{x}), &&	\hat{\Psi}(\hat{x})=\Psi(k(\hat{x})).
	\end{align*}
	Regarding the Hartree potential, one has to account for the fact that the transformation $k$ is only locally defined as a $C^2$ map,
	so the transform to a global integral operator is not well-defined. However, it is possible to evaluate $V_H(\hat{\Psi})(\hat{x})$ as $V_H(\Psi)(x)$ in $x=k(\hat{x})$.
	
	With this preparation, let $U$ be the image of a half-ball under $k$. Then we bound $\|\Psi\|_{H^1(U)}\leq C(\|F\|_{L^2(\Omega)}+\|\Psi\|_{L^2(\Omega)})$. As $\Omega$ is compact, it can be covered with finitely many sets $U_i$, so we find
	\begin{align*}
		\|\Psi\|_{H^1(\Omega)}\leq \sum_i \|\Psi\|_{H^1(U_i)} \leq \sum_i C\left(\|F\|_{L^2(\Omega)}+\|\Psi\|_{L^2(\Omega)}\right).
	\end{align*}
	
	Now the standard proof for elliptic equations based on difference quotients can be applied, e.g., \cite[§6.3.2, Theorem 4]{Evans2010}.
\end{proof}

\begin{theorem}\label{thm:ImprovedRegularity}
	Assume $\Psi_0\in H_0^1(\Omega)$, $F\in Y$ and $\partial \Omega \in C^2$. Suppose $\Psi \in W$ is the solution to \eqref{eq:KSweak}.
	Then
	\begin{align*}
		\Psi \in L^2(0,T;H^2(\Omega; \C^N)) \cap L^\infty(0,T; H_0^1(\Omega; \C^N)), \quad \Psi'\in L^2(0,T;L^2(\Omega; \C^N)).
	\end{align*}
	Furthermore the following estimate holds
	\begin{align}\label{thm:ImprovedRegularity:Estimate}
		\operatorname*{ess\, sup}_{0\leq t \leq T} \|\Psi(t)\|_{H^1}+\|\Psi \|_{L^2(0,T;H^2(\Omega))}+\|\Psi'\|_Y \leq C\left( \|\Psi_0\|_{H^1}+ \|F\|_Y \right).
	\end{align}
\end{theorem}

\begin{proof}
	We recall \eqref{eq:H1estimate}, that is,
	\begin{align}\label{eq:Improved:Psim}
		\operatorname*{ess\, sup}_{0\leq t \leq T} \|\Psi(t)\|_{H^1(\Omega)}^2\leq C(\|\Psi_0\|_{L^2(\Omega)}^2+\|F\|_Y^2),
	\end{align}
	which means that $\Psi\in L^\infty(0,T;H_0^1(\Omega)$).
	
	For $\pa_t\Psi$, we consider the Galerkin space $W_m$ and take a fixed $m$, multiply \eqref{eq:KSweak_test} with $\pa_td_m^k(t)$, and sum for $k=1,\dotsc,m$ to obtain the following
		\begin{equation}\label{eq:improved:test}\begin{split}
			\scalar{\pa_t\Psi_m(t)}{\pa_t\Psi_m(t)}&=B(\Psi_m(t), \pa_t\Psi_m(t);u(t))+\alpha\scalar{V(\Psi(t))\Psi(t)}{\pa_t\Psi_m(t)}\\
			&\quad+\scalar{F(t)}{\pa_t\Psi_m(t)}
		\end{split}\end{equation}
		a.e. in $(0,T)$. 
	For $D(\Psi_m(t),\pa_t\Psi_m(t))$, we have
	\begin{align*}
		|D_{xc}(\Psi_m(t),\pa_t \Psi_m(t))|&=\left|2\scalar{\pabl{V_{xc}}{\rho}(\Lambda(t)) \Re\scalarC{\Psi_m(t)}{\Lambda(t)}\Lambda(t)}{\pa_t\Psi_m(t)}\right|.
	\end{align*}
	Because $\Lambda(t)\in L^\infty(\Omega)$, we have
	\begin{equation}\label{eq:Imp:Dxc}\begin{split}
		|D_{xc}(\Psi_m(t),\pa_t \Psi_m(t))|&\leq C \int_\Omega \left|\sum_{i=1}^N \Re \psi_{i,m}(t) \sum_{j=1}^N \co{\pa_t \psi_{j,m}(t)} \right| \dd x\\
		&\leq C \int_\Omega \sum_{i,j=1}^N \frac{1}{\epsilon}|\psi_{i,m}(t)|^2+\epsilon |\pa_t\psi_{j,m}(t)|^2\\
		&\leq C N (\frac{1}{\epsilon}\|\Psi_m(t)\|_{L^2}^2+\epsilon\|\pa_t \Psi_m(t)\|_{L^2}^2),
	\end{split}\end{equation}
	where we use  Young's inequality for products.
	For $D_H$, we use Young's inequality for convolutions \cite[Theorem 14.6]{Schilling2005} and the fact that $\Lambda \in L^\infty(\Omega)$. We have
	\begin{equation}\label{eq:Imp:DH}\begin{split}
		|D_H(\Psi_m(t),\pa_t \Psi_m(t))|&=\left|\sum_{j=1}^N\int_\Omega (2\Re \scalarC{\Psi_m(t)}{\Lambda(t)}\star w)(x)\Lambda_j(x,t) \co{\pa_t\Psi_{m,j}(x,t)} \dd x \right|\\
		&\leq C'\scalar{|\Psi_m(t)|\star w}{|\pa_t\Psi_m(t)|}\\
		&\leq C' \| |\Psi_m(t)|\star w\|_{L^2} \|\pa_t\Psi_m(t)\|_{L^2}\\
		&\leq C' \| \Psi_m(t)\|_{L^2}\| w\|_{L^1} \|\pa_t\Psi_m(t)\|_{L^2},
	\end{split}\end{equation}
	where $w$ represents the Coulomb potential.
	Consequently, by \eqref{eq:Imp:Dxc} and \eqref{eq:Imp:DH}, we get the following
	\begin{equation}\label{eq:Imp:D}
			|D(\Psi_m(t),\pa_t \Psi_m(t))|\leq c (\frac{1}{\epsilon}\|\Psi_m(t)\|_{L^2}^2+\epsilon\|\pa_t \Psi_m\|_{L^2}^2) =: \tilde{D}.
	\end{equation}

	Estimate \eqref{eq:Imp:DH} is used together with \eqref{eq:improved:test} to obtain the following
		\begin{align*}
			&\|\pa_t\Psi_m(t)\|_{L^2}^2
			\\&\leq \frac{1}{2} \abl{}{t}\scalar{\nabla\Psi_m(t)}{\nabla\Psi_m(t)}+\scalar{V_{ext}(t) \Psi_m(t)}{\pa_t\Psi_m(t)}+D(\Psi_m(t),\pa_t\Psi_m(t))\\
			&\quad+\scalar{V(\Psi_m(t))\Psi_m(t)}{\pa_t\Psi_m}+\scalar{F(t)}{\pa_t\Psi_m(t)}\\
			&\leq \frac{1}{2} \abl{}{t}\scalar{\nabla\Psi_m(t)}{\nabla\Psi_m(t)}+\|V_{ext}(t)\|_{L^\infty} \|\Psi_m(t)\|_{L^2}\|\pa_t\Psi_m(t)\|_{L^2} + \tilde{D}\\
			&\quad +c_u\|\Psi_m(t)\|_{L^2}\|\Psi_m(t)\|_{H^1}^2\|\pa_t\Psi_m(t)\|_{L^2} +K \|\Psi_m(t)\|_{L^2}\|\pa_t\Psi_m(t)\|_{L^2}\\
			&\quad +\|F(t)\|_{L^2}\|\pa_t\Psi_m(t)\|_{L^2},
		\end{align*}
		where we use Lemma \ref{CancesLemma} and Assumptions \ref{assumptions} (\ref{assumptionVcBounded}) and (\ref{assumptionVxLipschitz}) to estimate $V$.
		Next, by using Cauchy-Schwarz inequality, \eqref{eq:H1estimate} and Young's inequality with an arbitrary positive $\epsilon$, we get
		\begin{align*}
			&\|\pa_t\Psi_m(t)\|_{L^2}^2\\
			&\leq \frac{1}{2} \abl{}{t} \|\Psi_m(t)\|_{H^1}^2+ \frac{1}{\epsilon}\|F(t)\|_{L^2}^2+\epsilon\|\pa_t\Psi_m(t)\|_{L^2}^2 + \tilde{D}\\
			&\quad + 			(\|V_{ext}(t)\|_{L^\infty}+K+C_1(\|\Psi_0\|_{L^2}^2+\|F\|_Y^2))\left(\frac{1}{\epsilon}\|\Psi_m(t)\|_{L^2}^2+\epsilon\|\pa_t\Psi_m(t)\|_{L^2}^2 \right)\\
			&\leq\frac{1}{2} \abl{}{t}\|\Psi_m(t)\|_{H^1}^2+\frac{\Gamma}{\epsilon}\left( \|\Psi_m(t)\|_{L^2}^2+\|F(t)\|_{L^2}^2\right)
			+\epsilon \Gamma \|\pa_t\Psi_m(t)\|_{L^2}^2,
		\end{align*}
		where $\Gamma$ is a constant depending only on $\|\Psi_0\|_{L^2}$, $\|F\|_Y$, $\max_{t\in [0,T]}\|V_{ext}(t)\|_{L^\infty}$ and $K$.
	Now, we choose $\epsilon$ small enough, that is $\epsilon < \frac{1}{\Gamma}$ and integrate from $0$ to $T$. We obtain
	\begin{align*}
			\int_0^T \|\pa_t\Psi_m(t)\|_{L^2}^2\dd t &\leq \frac{1}{1-\epsilon \Gamma}\left( \operatorname*{ess\,sup}_{0\leq t\leq T} \|\Psi_m(t)\|_{H^1}^2+\frac{\Gamma}{\epsilon}\int_0^T \|\Psi_m(t)\|_{L^2}^2+\|F(t)\|_{L^2}^2 \dd t \right).
		\end{align*}
	
	Using \eqref{eq:estimate-1} and \eqref{eq:H1estimate}, this gives
	\begin{align}\label{eq:Improved:Psimprime}
		\|\pa_t\Psi_m\|_Y^2\leq \Gamma'(\|\Psi_0\|_{L^2}^2+\|F\|_Y^2).
	\end{align}
	Passing to the limit as $m\rightarrow \infty$ we find $\Psi'\in Y$.
	
	Now, we rewrite \eqref{eq:KSweak} for a fixed time $t$ as follows
	\begin{align}
		B(\Psi(t), \Phi; u(t))+\alpha\scalar{V(\Psi(t))\Psi(t)}{\Phi}=\scalar{-F(t)+i\pa_t \Psi(t)}{\Phi},
	\end{align}
	where $\Psi$ is the solution to \eqref{eq:KSweak}.
	Using Theorem \ref{thm:estimates} we have that the solution is bounded and, therefore, the estimate in Lemma \ref{lem:EllipticRegularity} holds. We have
	\begin{align}\label{eq:ellipticH2bound}
		\|\Psi(t)\|_{H^2} &\leq c (\|A(t)\|_{L^2}+\|\Psi(t)\|_{L^2})\\
		&\leq c(\|F(t)\|_{L^2}+\|\Psi'(t)\|_{L^2}+\|\Psi(t)\|_{L^2}),\label{eq:ellipticH2bound2}
	\end{align}
	where $A(t)=-F(t)+i\pa_t \Psi(t)$. Next, we integrate \eqref{eq:ellipticH2bound} from $0$ to $T$, and use \eqref{eq:estimate-2} and \eqref{eq:Improved:Psimprime} to obtain the following
	
	\begin{align*}
		\|\Psi\|_{L^2(0,T; H^2(\Omega))}^2
		&\leq C(\|\Psi_0\|_{L^2}^2+\|F\|_Y^2).
	\end{align*}
	All together, we have shown the estimate.
\end{proof}

\begin{theorem}\label{thm:ImprovedRegularity2}
	If in addition to the assumptions of Theorem \ref{thm:ImprovedRegularity}, $\Psi_0\in H^2(\Omega)\cap H_0^1(\Omega)$, and $F\in H^1(0,T;L^2(\Omega))$ hold, then for the solution of \eqref{eq:KSweak}, we have
	\begin{align}
		\Psi'\in L^\infty(0,T;L^2(\Omega)) && \text{and} && \Psi\in L^\infty(0,T;H^2(\Omega)).
	\end{align}
\end{theorem}
\begin{proof}
	Take a fixed $m\geq 1$. Differentiate \eqref{eq:KSweak_test} with respect to $t$, multiply this equation with $\pa_t d_m^k(t)$, sum over $k$, and integrate over $t$ to obtain
	\begin{equation}\label{eq:imp2:weakint}	\begin{split}
		\int_0^T i\scalar{\pa_t^2 \Psi_m}{\pa_t \Psi_m} \dd t&= \int_0^T B(\pa_t\Psi_m, \pa_t \Psi_m; u)+\scalar{V_u \pabl{}{t} \left( u\Psi_m\right)}{\pa_t\Psi_m}\\+\scalar{\pa_t F}{\pa_t\Psi_m}
		&+\alpha\scalar{V(\Psi_m)\pa_t \Psi_m+\pabl{V(\Psi_m)}{t}\Psi_m}{\pa_t\Psi_m} \dd t.
	\end{split}\end{equation}
	For the left-hand side, we have
	\begin{align}\label{eq:imp2:lhsnorm}
		i\scalar{\pa_t^2 \Psi_m}{\pa_t \Psi_m}=i\frac{1}{2}\abl{}{t} \|\pa_t\Psi_m\|_{L^2}^2.
	\end{align}
	We remark that for any $f(\Psi, x, t)\in \R$, we have
	\begin{align*}
		\scalar{f(\Psi, \cdot, t) \Psi(t)}{\pa_t \Psi(t)}=\scalar{f(\Psi, \cdot, t)}{\scalarC{\Psi}{\pa_t\Psi}}\\
		=\scalar{f(\Psi, \cdot, t)}{\frac{1}{2} \abl{}{t} \|\Psi(t)\|_\C^2} \in \R.
	\end{align*}
	Hence, using this result for the product terms in \eqref{eq:imp2:weakint}, we get
	\begin{align}\label{eq:imp2:realproducts}
		\scalar{V_u \pabl{}{t}\left(u\Psi_m(t)\right)}{\pa_t\Psi_m}\in \R && \text{and} &&
				\scalar{ \pabl{}{t}\left( V(\Psi_m)\Psi_m \right)}{\pa_t\Psi_m} \in \R.
	\end{align}
	
	Taking the imaginary part of \eqref{eq:imp2:weakint} and using \eqref{eq:imp2:lhsnorm} and \eqref{eq:imp2:realproducts} gives the following
	\begin{align*}
		\frac{1}{2}( 
		 \|\pa_t\Psi_m(t)\|_{L^2}^2-\|\pa_t\Psi_m(0)\|_{L^2}^2)&=\int_0^t (1-\alpha)\Im D(\pa_t\Psi_m, \pa_t\Psi_m)+\Im \scalar{\pa_t F}{\pa_t\Psi_m}.
	\end{align*}
		From this, using \eqref{eq:lem3:estD}, we obtain the following
		\begin{align*}
			\sup_{0\leq t\leq T}\|\pa_t \Psi_m(t)\|_{L^2}^2&\leq \|\pa_t\Psi_m(0)\|_{L^2}^2+2\int_0^T(1-\alpha)|\Im D(\pa_t\Psi_m, \pa_t\Psi_m)|\\
			&\quad+|\Im \scalar{\pa_t F}{\pa_t\Psi_m}| \dd t\\
					&\leq \|\pa_t\Psi_m(0)\|_{L^2}^2+2\int_0^T(1-\alpha)c_0 \|\pa_t\Psi_m\|_{L^2}^2+\|F\|_{L^2}^2+\|\pa_t \Psi_m\|_{L^2}^2 \dd t\\
					&\leq\|\pa_t\Psi_m(0)\|_{L^2}^2+2(c_0+1)\|\pa_t\Psi_m\|_Y^2 +2\|F\|_Y^2.
		\end{align*}
	
	By \eqref{thm:ImprovedRegularity:Estimate}, $\|\pa_t\Psi_m\|_Y$ is bounded by $F$ and $\Psi_0$. Hence, there exists a constant $c_6$ depending only on $T$, $\|\Psi_0\|_{L^2}$, $\|F\|_Y$ and $\|u\|_{H^1(0,T)}$, such that the following holds
	\begin{align}\label{eq:Imp2:supwithpsim0}
		\sup_{0\leq t\leq T}\|\pa_t \Psi_m(t)\|_{L^2}^2&\leq \|\pa_t\Psi_m(0)\|_{L^2}^2+ c_6.
	\end{align}
	
	To bound $\|\pa_t\Psi_m(0)\|_{L^2}^2$, we test \eqref{eq:KSweak_test} with $\pa_t\Psi_m(0)$ to obtain
	\begin{align*}
		i\scalar{\pa_t\Psi_m(0)}{\pa_t\Psi_m(0)}&=B(\Psi_m(0),\pa_t \Psi_m(0);u)+\scalar{V(\Psi_m(0)) \Psi_m(0)}{\pa_t \Psi_m(0)}\\
		&\quad +\scalar{F(0)}{\pa_t \Psi_m(0)},
	\end{align*}
	\begin{equation}\label{eq:improvedinfty:dtpsizero}
	\begin{split}
		\|\pa_t\Psi_m(0)\|_{L^2}^2 &\leq |B(\Psi_m(0),\pa_t \Psi_m(0);u)|+|\scalar{V(\Psi_m(0)) \Psi_m(0)}{\pa_t \Psi_m(0)}|\\
		&\quad +\|F(0)\|_{L^2}\|\pa_t\Psi_m(0)\|_{L^2}\\
		&\leq c_1' \|\Psi_m(0)\|_{H^2}\|\pa_t\Psi_m(0)\|_{L^2}+K'\|\Psi_m(0)\|_{L^2}\|\pa_t \Psi_m(0)\|_{L^2}\\
		&\quad +\|F(0)\|_{L^2}\|\pa_t\Psi_m(0)\|_{L^2}.
	\end{split}\end{equation}
	Here, we used \eqref{eq:bound_pot2} for the nonlinear potential and we use the modified proof of Lemma \ref{lem:bound-weak-form} by replacing $\scalar{\nabla \Psi}{\nabla \Phi}$ by $\scalar{\nabla^2 \Psi}{\Phi}$ using integration by parts.
	Dividing by $\|\pa_t\Psi_m(0)\|_{L^2}$ gives
	\begin{align*}
		\|\pa_t\Psi_m(0)\|_{L^2} &\leq c_1' \|\Psi_m(0)\|_{H^2}+K'\|\Psi_m(0)\|_{L^2}+\|F(0)\|_{L^2}\\
		&\leq (c_1'+K) \|\Psi_m(0)\|_{H^2}+\|F(0)\|_{L^2}.
	\end{align*}
	
	Furthermore, we have $\|\Psi_m(0)\|_{H^2}\leq C\|\Psi_0\|_{H^2}$; see, e.g., \cite[p. 363]{Evans2010}. Using this in \eqref{eq:improvedinfty:dtpsizero} gives the following
	\begin{align}\label{eq:Imp2:bla}
		\|\pa_t\Psi_m(0)\|_{L^2} &\leq (c_1'+1) C\|\Psi_0\|_{H^2}+\|F(0)\|_{L^2}.
	\end{align}
	
	Therefore, using \eqref{eq:Imp2:bla} in \eqref{eq:Imp2:supwithpsim0}, we obtain the following
	\begin{align*}
		\sup_{0\leq t \leq T} \|\pa_t \Psi_m(t)\|_{L^2}^2 \leq c_7 \left( \|\Psi_0\|_{H^2}^2+\|F(0)\|_{L^2}^2 \right)+c_6.
	\end{align*}
	
	Taking the limit $m\rightarrow \infty$, we find $\Psi'\in L^\infty(0,T; L^2(\Omega))$.
	
	Using this result in \eqref{eq:ellipticH2bound2}, we have that $\Psi\in L^\infty(0,T;H^2(\Omega))$ and  $\Psi$ is globally bounded by a constant $c_8$ depending on $T$, $\|\Psi_0\|_{L^2}$, $\|F\|_Y$ and $\|u\|_{H^1(0,T)}$ as follows
	\begin{equation}\label{eq:globalbound}
		\operatorname*{ess\,sup}_{0\leq t\leq T} \max_{x\in \Omega} |\Psi(x,t)|\leq c_8.
	\end{equation}
\end{proof}

\begin{remark}\label{remark:aposteriori}
By \eqref{eq:globalbound}, the solution of \eqref{eq:KSweak} is everywhere and for almost all times bounded by a constant.
As $V_x(\Psi)\Psi$ is a convex function of $\Psi$, it is hence Lipschitz continuous for solutions of \eqref{eq:KSweak}.
Assumption \ref{assumptions} (\ref{assumptionVxLipschitz}) is hence a reasonable assumption as it holds for all solutions.
\end{remark}

\section{Conclusion}
In this paper, the existence, uniqueness and improved regularity of solutions to the time-dependent Kohn-Sham (KS) equations and related equations were proved.
These results were proved considering a representative class of KS potentials. This work is instrumental for investigating optimal control problems governed by the KS equations.

\bibliographystyle{siam}
\bibliography{Literatur}

\end{document}